\documentclass[10pt,a4paper]{amsart}

\usepackage[british]{babel}
\usepackage[T1]{fontenc}
\usepackage[utf8]{inputenc}

\usepackage{amsmath,amssymb,amsfonts,amsthm,mathtools}
\usepackage{mathrsfs}

\usepackage{array}
\usepackage[a4paper,margin=1in]{geometry}

\newif\iflegacyfonts
\IfFileExists{newtxtext.sty}{%
  \IfFileExists{newtxmath.sty}{\usepackage{newtxtext,newtxmath}}%
    {\legacyfontstrue}%
}{\legacyfontstrue}
\iflegacyfonts
  \usepackage{mathptmx}
  \DeclareSymbolFont{CMlargesymbols}{OMX}{cmex}{m}{n}
  \let\coprod\relax
  \DeclareMathSymbol{\coprod}{\mathop}{CMlargesymbols}{"60}
\fi
\usepackage{microtype}

\usepackage{hyperref}
\hypersetup{hidelinks,
  pdftitle={Foliated and Mather--Jacobian discrepancies via arc spaces},
  pdfauthor={Maurício Corrêa}}
\newcommand{\thmref}[1]{\hyperref[#1]{Theorem~\ref*{#1}}}

\newcommand{\secref}[1]{\hyperref[#1]{Section~\ref*{#1}}}

\numberwithin{equation}{section}
\theoremstyle{plain}
\newtheorem{Theorem}{Theorem}[section]
\newtheorem{Lemma}[Theorem]{Lemma}
\newtheorem{Proposition}[Theorem]{Proposition}
\newtheorem{Corollary}[Theorem]{Corollary}

\theoremstyle{definition}
\newtheorem{Definition}[Theorem]{Definition}
\newtheorem{Example}[Theorem]{Example}
\newtheorem{Assumption}[Theorem]{Assumption}

\newcommand{\Spec}{\operatorname{Spec}}

\newcommand{\Supp}{\operatorname{Supp}}

\newcommand{\codim}{\operatorname{codim}}
\newcommand{\ord}{\operatorname{ord}}

\newcommand{\mld}{\operatorname{mld}}

\newcommand{\Nklt}{\operatorname{Nklt}}
\newcommand{\Nlc}{\operatorname{Nlc}}

\newcommand{\im}{\operatorname{Im}}

\def\bibaut#1{{\sc #1}}

\title[Foliated and Mather--Jacobian discrepancies]{Foliated and Mather--Jacobian discrepancies via arc spaces}
\author[M. Corr\^ea]{Maur\'icio Corr\^ea}
\address[Maur\'icio Corr\^ea]{Dipartimento di Matematica, Universit\`a degli Studi di Bari Aldo Moro, Via E. Orabona 4, I-70125 Bari, Italy}
\email{mauricio.barros@uniba.it}

\begin{document}
\begin{abstract}
We express the log discrepancies of co-rank one foliations on
normal varieties in terms of arc codimensions and differential orders.
On smooth threefolds, if a foliation has simple singularities
and its separatrices form an algebraic simple normal crossings
divisor, we identify its foliated discrepancies with those of an
ordinary pair for every divisor centred on the separatrices. This gives adjunction and log canonical threshold formulae.
For algebraic images of the separatrices, we estimate changes
in boundary orders caused by failure of ideal generation
and by descent through the conductor. For locally complete intersection images, comparison on a
common resolution yields criteria for equality of discrepancies
and bounds on log canonical thresholds.
\end{abstract}

\maketitle

\section{Introduction}

Divisorial valuations are represented by maximal divisorial
sets in arc spaces \cite{DL,MustataJets}. On a smooth variety, the
codimension of such a set gives the log discrepancy \cite{EMY}; on a
singular variety, the formula involves Mather discrepancies and
the Jacobian ideal \cite{dFEI,IshiiMather,EI,dFD}. For a
foliation, one must also account for the saturation of the
pulled-back defining one-form and whether the divisor is invariant.

Let $X$ be a normal complex variety of dimension $n\geq2$, and let
$\mathcal F\subset T_X$ be a co-rank one foliation, namely a saturated
subsheaf of rank $n-1$ closed under the Lie bracket. Let $\Delta$ be a
$\mathbb Q$-divisor such that $K_{\mathcal F}+\Delta$ is
$\mathbb Q$-Cartier. On a birational model $\mu:Y\to X$, write
$\mathcal G$ for the transformed foliation. For a prime divisor
$E\subset Y$, the foliated log discrepancy is
$A_{\mathcal F}(E;X,\Delta)=a(E;\mathcal F,\Delta)+\epsilon(E)$,
where $a(E;\mathcal F,\Delta)$ is the coefficient of $E$ in
$K_{\mathcal G}-\mu^*(K_{\mathcal F}+\Delta)$, and
$\epsilon(E)$ is $0$ if $E$ is invariant and $1$ otherwise
\cite{CSMMP}. We use the same convention for divisors already on
$X$.

For $v=q\operatorname{ord}_E$, with $q\geq1$, let $C_X(v)$ denote the
maximal divisorial set: the closure in $J_\infty(X)$ of the image of
the contact-$q$ arcs along $E$ on $Y$ \cite{dFEI,ELM}.
Write $\gamma_v$ for its generic arc. Choose $r>0$ such that
$\mathcal L_r=\mathcal O_X(r(K_{\mathcal F}+\Delta))$ is
invertible. The generic conormal exact sequence of $\mathcal F$
induces a homomorphism with image
\[
\mathcal P_r(\mathcal F,\Delta)
 =\operatorname{im}\left[
 (\textstyle\bigwedge^n\Omega_X^1)^{\otimes r}\otimes
 \mathcal L_r^{-1}\longrightarrow
 (\Omega^1_{\mathbb C(X)/\mathbb C})^{\otimes r}
 \right].
\]
Along $\gamma_v$, a local section of this image becomes a rational
multiple of $(dt)^{\otimes r}$. We set
\[
\vartheta_{\mathcal F,\Delta}(v)
 =\frac{1}{r}\operatorname{ord}_{\gamma_v}
      \mathcal P_r(\mathcal F,\Delta).
\]
Here the order is the minimum of the orders of local generators.
In Section~\ref{sec:module} we prove finiteness and independence of
$r$.

\begin{Theorem}\label{thm:intro}
For every divisorial valuation $v=q\operatorname{ord}_E$, the order
$\vartheta_{\mathcal F,\Delta}(v)$ is finite and independent of
$r$. Moreover,
\begin{align}
A_{\mathcal F}(v;X,\Delta)
 &=\operatorname{codim}C_X(v)
   +\vartheta_{\mathcal F,\Delta}(v)
   -\vartheta_{\mathcal F,\Delta}(2v),
   \label{eq:intro-general}\\
\vartheta_{\mathcal F,\Delta}(2v)
 -2\vartheta_{\mathcal F,\Delta}(v)
 &=\delta_E,
\end{align}
where $\delta_E=1$ for an invariant divisor and $0$ otherwise,
and $A_{\mathcal F}(q\operatorname{ord}_E;X,\Delta)
=qA_{\mathcal F}(E;X,\Delta)$.
\end{Theorem}

When $X$ is smooth, the formula takes a simpler form. Let $\omega$
be a local saturated one-form defining $\mathcal F$, and write
$t_{\mathcal F}(v)$ for the order of $\gamma_v^*\omega$. Then
\begin{equation}\label{eq:intro-smooth}
A_{\mathcal F}(v;X,\Delta)
 =\operatorname{codim}C_X(v)
  +t_{\mathcal F}(v)-t_{\mathcal F}(2v)-v(\Delta).
\end{equation}
The order of the pulled-back form therefore distinguishes
invariant from non-invariant divisors. For a fixed prime divisor $E$
and $r$ as above, we also obtain
\begin{equation}\label{eq:intro-index}
A_{\mathcal F}(E;X,\Delta)\geq0
\quad\Longleftrightarrow\quad
\vartheta_{\mathcal F,\Delta}((r+1)\operatorname{ord}_E)
 \leq\operatorname{codim}C_X((r+1)\operatorname{ord}_E).
\end{equation}
Log canonicity requires the inequality for every divisor over $X$.

We then consider simple singularities of foliations on smooth
threefolds. Cano's normal forms \cite{Can04} allow us to describe
the formal arcs tangent to the foliation, where tangency means $\gamma^*\omega=0$ for a local
defining one-form $\omega$. By Theorem~\ref{thm:confinement}, every
such arc centred at a simple point lies in the reduced union of the
formal coordinate separatrices. For the first simple type, the
proof uses positive non-resonance; for the second, it uses the leading monomial.
We discuss finite tangent jets \cite{Carter} in
Section~\ref{sec:tangential}.

Consider a foliation $\mathcal H$ with simple singularities on a
smooth threefold $W$. Suppose that its separatrices form an algebraic
SNC divisor $D$, as in Section~\ref{sec:systems}. For every prime
divisor $E$ over $W$ centred in $D$ and every rational divisor $B$,
Theorem~\ref{thm:ordinary-pair} gives
\[
 A_{\mathcal H}(E;W,B)=A(E;W,D+B).
\]
This equality allows us to apply ordinary adjunction on each
separatrix. The foliated different on a component $S_i$ is the
boundary induced by the other separatrices and by $B$, together
with an explicit effective divisor $R_i$.
The discrepancy equality is preserved under birational morphisms
between smooth threefolds whose exceptional centres lie in $D$,
even if the total transform of $D$ is not SNC
(Theorem~\ref{thm:excess-independence}). Furthermore, every
divisorial valuation on a normal strict transform of a separatrix
admits ambient invariant divisors preserving its discrepancy and
the orders of any prescribed finite collection of ideals
(Theorem~\ref{thm:regular-leaf-lift}).

Applying ordinary inversion of adjunction \cite{EMY}, we obtain
the equality of minimal log discrepancies with ideals in
Theorem~\ref{thm:ideal-adjunction}.
If a divisor over a branch computes a finite log canonical
threshold, infinitely many ambient invariant divisors compute
the same threshold, with unbounded order along the branch.
Example~\ref{ex:weak-threshold} computes such a threshold on a
weak separatrix and shows how the term $R_i$ enters the formula.

The branch discrepancy formulae follow from adjunction
\cite{CSMMP,CSAdj}. When the ordinary boundary is effective,
minimal log discrepancies with non-negative ideal weights are
lower semicontinuous along $D$ (Corollary~\ref{cor:ideal-lsc});
the analogous assertion for proper finite systems of SNC pairs
is proved later in the paper.

We also consider algebraic images of the separatrices. These
images may be singular, and their boundary orders may differ from
those on a smooth branch even after normalisation. The image
ideal may fail to generate the branch ideal, and ideals on different components may fail
to descend exactly through the conductor. Section~\ref{sec:mj}
treats these differences separately. Even when both vanish,
Mather--Jacobian discrepancies on a singular image need not agree
with ordinary discrepancies on a resolution. For locally complete
intersection images, the comparison also involves the relative
canonical divisor. Proposition~\ref{prop:lci-threshold-comparison}
gives a criterion for equality and bounds on log canonical thresholds;
Example~\ref{ex:lci-generation-cancellation} shows how the terms
can cancel on a quadric cone.

\section{Differential orders on smooth varieties}
\label{sec:smooth}

We begin with the order of a saturated defining one-form
along a divisorial arc on a smooth variety.

\subsection{Valuations and generic orders}

We work over $\mathbb C$. The ambient varieties are irreducible;
the image strata in Section~\ref{sec:mj} may be reducible.
A co-rank one foliation $\mathcal F$ is a saturated involutive subsheaf
of $T_X$ of rank $n-1$. Its canonical sheaf is
$(\bigwedge^{n-1}\mathcal F^*)^{**}$, represented by $K_{\mathcal F}$.
If $\mu:Y\to X$ is a smooth birational model and $\mathcal G$ is the
transformed foliation, define
\[
a(E;\mathcal F,\Delta)
 =\operatorname{coeff}_E
       \bigl(K_{\mathcal G}-\mu^*(K_{\mathcal F}+\Delta)\bigr),
\qquad
A_{\mathcal F}(E)=a(E;\mathcal F,\Delta)+\epsilon(E),
\]
where compatible canonical divisors are used,
$\epsilon(E)=0$ for an invariant divisor and $1$ otherwise.
This formula includes divisors already on $X$.
We put $\delta_E=1-\epsilon(E)$ and extend log discrepancies
homogeneously to positive integral multiples of $\operatorname{ord}_E$.
The same convention is used on subpairs with negative coefficients.
Log canonicity is tested on all prime divisors over $X$.

For $v=q\operatorname{ord}_E$, the maximal divisorial set $C_X(v)$
is the closure of the images of arcs on $Y$ having contact $q$
with $E$ at its generic point.
The set is irreducible and independent of the chosen model.
Its generic arc is fat: it does not factor through a proper
closed subvariety of $X$. The codimension formula is
\begin{equation*}
\operatorname{codim}C_X(q\operatorname{ord}_E)
 =q(\widehat k_E(X)+1).
\end{equation*}
For smooth $X$, $\widehat k_E(X)=\operatorname{ord}_E K_{Y/X}$.
For normal singular $X$, the left side is the codimension of a
maximal divisorial set computed from high jet truncations, and
$\widehat k_E$ is the Mather discrepancy \cite{ELM,dFEI}.

Suppose for the moment that $X$ is smooth. Its conormal sheaf
$N^*_{\mathcal F}$ is a line bundle embedded in $\Omega_X^1$.
A local saturated generator $\omega$ defines
$t_{\mathcal F}(v)=\operatorname{ord}_t\gamma^*\omega$
at the generic arc of $C_X(v)$.
The order does not depend on the generator. We calculate it at
the generic point of the divisor.

\begin{Lemma}\label{lem:profile}
At the generic point of $E$ write
$\mu^*\omega=u^b\alpha$, up to a unit, where $\alpha$ generates
$N^*_{\mathcal G}$. Then
\[
t_{\mathcal F}(q\operatorname{ord}_E)
 =qb+(q-1)\delta_E.
\]
\end{Lemma}
\begin{proof}
We may work on a smooth open subset of $E$ on which $\mathcal G$
is regular. Choose local parameters $u,z_2,\ldots,z_n$ with $E=(u=0)$.
Write
\[
\alpha=a\,du+\sum_{i=2}^n b_i\,dz_i.
\]
If $E$ is invariant, every $b_i$ is divisible by $u$, and $a$
is a unit at its generic point. For a generic contact-$q$ arc,
$u(t)=t^q c(t)$ with $c(0)\ne0$.
The first term has order $q-1$, with non-zero leading coefficient
$q\,a(0)c(0)$. All other terms have order at least $q$.

If $E$ is transverse, some $b_i$ is a unit modulo $u$.
The first derivatives of the $z_i(t)$ are independent generic
coefficients. Thus the coefficient of $dt$ in the pulled-back form
is generically non-zero at $t=0$. This remains true for $q=1$,
when the $du$ term can also contribute at order zero:
it is a non-zero linear expression in independent coefficients.
Thus $\gamma^*\alpha$ has order zero.

In both cases, multiplication by $u^b$ adds $qb$. The arcs used
above are dense in $C_X(v)$, which proves the formula for the
generic order.
\end{proof}

\begin{Theorem}
For a smooth $X$, every co-rank one foliation and every
$\mathbb Q$-divisor $\Delta$ satisfy \eqref{eq:intro-smooth}.
Moreover,
\[
t_{\mathcal F}(2v)-2t_{\mathcal F}(v)=\delta_E.
\]
\end{Theorem}
\begin{proof}
The saturated pull-back conormal bundle satisfies
$N^*_{\mathcal G}=\mu^*N^*_{\mathcal F}\otimes\mathcal O_Y(bE)$
at the generic point of $E$.
Since $K_{\mathcal F}=K_X-N^*_{\mathcal F}$ in additive notation,
\[
a(E;\mathcal F,\Delta)
 =\operatorname{ord}_E K_{Y/X}-b-\operatorname{ord}_E\Delta.
\]
Adding $\epsilon(E)=1-\delta_E$ gives
$A_{\mathcal F}(E)=A_X(E)-b-\delta_E-\operatorname{ord}_E\Delta$.
By Lemma~\ref{lem:profile},
$t_{\mathcal F}(2v)-t_{\mathcal F}(v)=q(b+\delta_E)$.
Substitute the ordinary codimension
$\operatorname{codim}C_X(v)=qA_X(E)$.
The second equality follows from the same lemma.
\end{proof}

The term $q-1$ arises by differentiating a local equation
of $E$ along a contact-$q$ arc. For $\Delta=0$,
$\operatorname{codim}C_X(v)-t_{\mathcal F}(v)=A_{\mathcal F}(v)+\delta_E$.
The order at multiplicity one alone need not detect log canonicity.

\section{A differential module on normal varieties}
\label{sec:module}

On a normal variety, the codimension of a divisorial set is
given by the Mather discrepancy. We calculate the corresponding
conormal order at the generic point of a divisor over $X$.

\subsection{Construction on the original space}

Choose $r>0$ such that
$\mathcal L_r=\mathcal O_X(r(K_{\mathcal F}+\Delta))$ is invertible.
The conormal sequence at the generic point gives the image
appearing in Theorem~\ref{thm:intro}:
\begin{equation}\label{eq:intro-module}
\mathcal P_r(\mathcal F,\Delta)
 =\operatorname{im}\left[
 (\textstyle\bigwedge^n\Omega_X^1)^{\otimes r}\otimes\mathcal L_r^{-1}
 \longrightarrow(\Omega^1_{\mathbb C(X)/\mathbb C})^{\otimes r}\right].
\end{equation}
Put $K=\mathbb C(X)$ and let $N^*_{\mathcal F,K}$ be the annihilator
of the generic foliation in $\Omega^1_{K/\mathbb C}$. The exact sequence
$0\to N^*_{\mathcal F,K}\to\Omega^1_{K/\mathbb C}
\to\mathcal F_K^*\to0$ gives
$\bigwedge^n\Omega^1_{K/\mathbb C}
\simeq N^*_{\mathcal F,K}\otimes\det(\mathcal F_K^*)$.
Cancel the leafwise determinant using $\mathcal L_r^{-1}$, with
the generic trivialisation of $\mathcal O_X(r\Delta)$, to obtain
the arrow in \eqref{eq:intro-module}. The image is a coherent
rank-one subsheaf of rational one-form tensors. Torsion in the
source maps to zero.

To describe the module locally, let $\eta$ be a rational top-degree
form and $\omega$ a rational generator of the conormal line.
The quotient $\eta/\omega$ represents a rational leafwise volume form:
contract $\eta$ with a transverse vector field $v$ such that
$\omega(v)=1$, and restrict to $\mathcal F$. This does not depend
on $v$ and represents $K_{\mathcal F}$.
Locally write $r(K_{\mathcal F}+\Delta)=\operatorname{div}(f)$ and
$\mathfrak a_\eta=\im[\bigwedge^n\Omega_X^1\xrightarrow{/\eta}K]$. Thus
\begin{equation}\label{eq:module-local}
\mathcal P_r(\mathcal F,\Delta)
 =f\mathfrak a_\eta^r\omega^{\otimes r}.
\end{equation}
If $(\eta,\omega)$ is replaced by $(g\eta,h\omega)$,
then $\mathfrak a_\eta$ becomes $g^{-1}\mathfrak a_\eta$
and $f$ becomes $f(g/h)^r$, up to a unit. The factors cancel
in \eqref{eq:module-local}, so the module is independent of
these choices.

\subsection{Orders along fat arcs}

Let $\gamma:\Spec L[[t]]\to X$ be a fat arc. Substituting it
into a rational one-form and differentiating with respect to $t$
gives an element of $L((t))dt$; tensor powers give elements of
$L((t))(dt)^{\otimes r}$. For a coherent subsheaf
$\mathcal M\subset(\Omega^1_{K/\mathbb C})^{\otimes r}$,
let $\ord_\gamma\mathcal M$ be the smallest coefficient order
among its local generators. This does not depend on the generators:
regular linear combinations cannot decrease the minimum. Negative
orders record poles, and $+\infty$ means that every generator
vanishes on the arc. For finite order, after a common denominator is
fixed, the condition is determined by finitely many coefficients.

\begin{Definition}
For the generic arc $\gamma_v$ of $C_X(v)$, put
$\vartheta_{\mathcal F,\Delta}(v)=r^{-1}\ord_{\gamma_v}
\mathcal P_r(\mathcal F,\Delta)$.
\end{Definition}
For a positive integer $m$, the determinant construction identifies
$\mathcal P_{mr}$ with the image of $\mathcal P_r^{\otimes m}$.
Its order is $m$ times that of $\mathcal P_r$; taking a common
multiple shows that the normalised order is independent of $r$.
The order is finite at the generic divisorial arc, but may be
infinite at special arcs of $C_X(v)$.

\subsection{The generic divisor calculation}

Choose a smooth model $\mu:Y\to X$ on which $E$ appears and which
factors through the Nash blow-up. Write
\[
\operatorname{im}\left[\mu^*\bigwedge^n\Omega_X^1
                  \longrightarrow\bigwedge^n\Omega_Y^1\right]
 =\mathcal O_Y(-\widehat K_{Y/X})\otimes\omega_Y.
\]
Put $\widehat k=\operatorname{coeff}_E\widehat K_{Y/X}$ and
$a=a(E;\mathcal F,\Delta)$.
At the generic point of $E$, the transformed foliation $\mathcal G$
is regular: its singular locus has codimension at least two.
Let $\alpha$ generate its saturated conormal line there.

\begin{Lemma}\label{lem:module-pullback}
At the generic point of $E$, the pull-back image of the differential
module is generated, up to a unit, by
\begin{equation}\label{eq:module-divisor}
u^{\,r(\widehat k-a)}\alpha^{\otimes r}.
\end{equation}
\end{Lemma}
\begin{proof}
The pull-back image of top differentials contributes
$\mathcal O_Y(-r\widehat K_{Y/X})\omega_Y^{\otimes r}$.
The saturated determinant sequence identifies
$\omega_Y^{\otimes r}\otimes\mathcal O_Y(-rK_{\mathcal G})$
with the line generated by $\alpha^{\otimes r}$. Finally,
$\mathcal O_Y(rK_{\mathcal G})\otimes\mu^*\mathcal L_r^{-1}$
has generator $u^{-ra}$ at $E$. Multiplication gives
$u^{r(\widehat k-a)}\alpha^{\otimes r}$. The exponent is integral
because $r(K_{\mathcal F}+\Delta)$ is Cartier.
\end{proof}

\begin{Proposition}\label{prop:general-profile}
For every prime divisor $E$ over $X$ and every $q\geq1$,
\begin{equation}\label{eq:general-profile}
\vartheta_{\mathcal F,\Delta}(q\operatorname{ord}_E)
 =q\bigl(\widehat k_E(X)-a(E;\mathcal F,\Delta)\bigr)
          +(q-1)\delta_E.
\end{equation}
Equivalently, writing
$\widehat A_X(q\operatorname{ord}_E)=q(\widehat k_E(X)+1)$,
\begin{equation*}
\vartheta_{\mathcal F,\Delta}(v)
 =\widehat A_X(v)-A_{\mathcal F}(v;X,\Delta)-\delta_E.
\end{equation*}
\end{Proposition}
\begin{proof}
Use Lemma~\ref{lem:module-pullback} on generic contact-$q$ arcs.
The argument of Lemma~\ref{lem:profile}, applied to the saturated
form $\alpha$ on $Y$, gives
$\operatorname{ord}_t\gamma^*\alpha=(q-1)\delta_E$.
The evaluated tensor in \eqref{eq:module-divisor} has order
$rq(\widehat k-a)+r(q-1)\delta_E$.
These arcs have dense image in $C_X(q\operatorname{ord}_E)$.
Generic image arcs lift uniquely because their generic points lie
in the isomorphism locus. Evaluation of rational tensors commutes
with this lift. Dividing by $r$ proves the first formula.
Since $A_{\mathcal F}(E)=a+1-\delta_E$, the second is the same equality.
\end{proof}

\begin{Corollary}
The generic arc of a maximal divisorial set is not tangent to the foliation.
\end{Corollary}
\begin{proof}
Its generic arc has finite module order by
Proposition~\ref{prop:general-profile}. Every module section is a
rational multiple of a tensor power of the conormal generator; a fat
arc sends a non-zero rational multiplier to a non-zero Laurent series.
Thus that generator cannot vanish on the generic arc.
\end{proof}

\begin{proof}[Proof of Theorem~\ref{thm:intro}]
Taking the difference of \eqref{eq:general-profile} at $2q$ and
$q$ gives
\[
\vartheta_{\mathcal F,\Delta}(2v)
 -\vartheta_{\mathcal F,\Delta}(v)
 =q\bigl(\widehat k_E(X)-a(E;\mathcal F,\Delta)+\delta_E\bigr).
\]
Subtract this from
$\operatorname{codim}C_X(v)=q(\widehat k_E(X)+1)$.
The result is $q(a+1-\delta_E)=A_{\mathcal F}(v)$.
Subtracting twice the expression at $q$ from the expression at $2q$
in \eqref{eq:general-profile} gives $\delta_E$.
\end{proof}

\subsection{Reparametrisation of divisorial arcs}

Choose a smooth open $E^\circ\subset E$ on a smooth Nash model on
which the foliation is regular and
$\mathcal P_r\mathcal O_Y=\mathcal O_Yu^{rB_E}\alpha^{\otimes r}$,
where $B_E=\widehat k_E-a(E;\mathcal F,\Delta)$.
Let $\mathscr U_{E,k}$ be the cylinder of arcs with origin in
$E^\circ$ and order $k$ along $E$. The normalised order of this generator is
$\Theta_E(\gamma)=kB_E+\ord_t\gamma^*\alpha$; its generic value is
$d_{E,k}=\vartheta_{\mathcal F,\Delta}(k\ord_E)$.

\begin{Proposition}
\label{prop:contact-filtration}
For $\ell\in\mathbb Z_{\geq0}$, put
$\mathscr U_{E,k}^{\geq\ell}=\{\Theta_E\geq d_{E,k}+\ell\}$.
If $E$ is invariant, this is $\mathscr U_{E,k}$ for $\ell=0$ and
empty otherwise. If $E$ is transverse, it is a non-empty finite-jet
cylinder of relative codimension $\ell$ in $\mathscr U_{E,k}$.
On fat arcs these conditions are those of the intrinsic module.
\end{Proposition}
\begin{proof}
For invariant $E$, write $\alpha=a\,du+u\sum b_jdz_j$, with
$a|_E\ne0$. Substituting $u=t^kc(t)$ gives order exactly $k-1$,
with leading coefficient $ka(\gamma(0))c(0)\ne0$.
For transverse $E$, some tangential coefficient, say that of $dz_2$,
is a unit. The successive coefficients of $\gamma^*\alpha$ can
then be solved recursively for the positive-degree coefficients
of $z_2(t)$: the equation at degree $j-1$ has coefficient
$j b_2(\gamma(0))$ on $z_{2,j}$. The first $\ell$ equations are
independent and compatible with the prescribed order of $u$.
They impose codimension $\ell$ and admit solutions. This argument
uses algebraic finite-jet equations and works locally along $E^\circ$.
\end{proof}

\begin{Proposition}
For $v=q\ord_E$ and $m\geq2$, the arc
$\gamma_v^{[m]}(t)=\gamma_v(t^m)$ belongs to $C_X(mv)$ and satisfies
\begin{equation*}
\frac1r\ord_{\gamma_v^{[m]}}\mathcal P_r
   -\vartheta_{\mathcal F,\Delta}(mv)=(m-1)\epsilon(E).
\end{equation*}
For transverse $E$ it lies in a higher-contact stratum of relative
codimension $m-1$ on the smooth model; for invariant $E$ its order
is generic.
\end{Proposition}
\begin{proof}
By the chain rule, the normalised order is
$m\vartheta(v)+m-1$. Proposition~\ref{prop:general-profile}
gives $\vartheta(mv)=m\vartheta(v)+(m-1)\delta_E$.
Subtract and use Proposition~\ref{prop:contact-filtration}.
\end{proof}

\begin{Example}
For $\mathcal F=\ker dx_1$ on $\mathbb A^n$, the locus
$\ord_t\gamma^*dx_1\geq\ell$ is defined by the first $\ell$
positive-degree coefficients of $x_1(t)$. It is an irreducible
codimension-$\ell$ cylinder, but its generic origin is the generic
point of $X$, so its valuation on $\mathbb C(X)$ is trivial.
Thus unrestricted differential-contact components need not be
maximal divisorial sets.
\end{Example}

\section{Birational transformation and log canonicity}

Although the module is defined on $X$, its order is computed on a
model carrying the valuation. The relative Mather divisor appears
when we pull the module back to that model.

\subsection{Change of model}

\begin{Theorem}
Let $\mu:Y\to X$ be a proper birational morphism from a smooth variety,
factoring through the Nash blow-up,
and let $\mathcal G$ be the transformed foliation. Define the full
crepant subboundary $\Delta_Y$ by
$
K_{\mathcal G}+\Delta_Y=\mu^*(K_{\mathcal F}+\Delta).
$
Then, as modules of rational differential tensors,
\begin{equation}\label{eq:module-crepant}
\mathcal P_r(\mathcal F,\Delta)\mathcal O_Y
 =\mathcal O_Y(-r\widehat K_{Y/X})
                   \mathcal P_r(\mathcal G,\Delta_Y).
\end{equation}
For a divisorial valuation $v$ centred on $Y$,
\begin{align}
\vartheta_{\mathcal F,\Delta}(v)
 &=\vartheta_{\mathcal G,\Delta_Y}(v)
                   +v(\widehat K_{Y/X}),\label{eq:contact-crepant}\\
\operatorname{codim}C_X(v)
 &=\operatorname{codim}C_Y(v)
                   +v(\widehat K_{Y/X}).\label{eq:codim-crepant}
\end{align}
Consequently the expression in \eqref{eq:intro-general} is
unchanged by this crepant transformation.
\end{Theorem}
\begin{proof}
The pull-back image of $\bigwedge^n\Omega_X^1$ is
$\mathcal O_Y(-\widehat K_{Y/X})\otimes\omega_Y$.
The crepant equality identifies
$\mu^*\mathcal L_r$ with
$\mathcal O_Y(r(K_{\mathcal G}+\Delta_Y))$.
Their generic determinant maps are the same map, since the function
field and its foliation are unchanged.
Taking the image of their tensor product proves
\eqref{eq:module-crepant} on $Y$.
Evaluation along generic arcs gives \eqref{eq:contact-crepant}.
For a higher smooth model $Z\to Y$ carrying $v=q\operatorname{ord}_E$,
composition of the top-differential images gives
$\widehat K_{Z/X}=K_{Z/Y}+h^*\widehat K_{Y/X}$.
The Mather codimension formula then gives
\eqref{eq:codim-crepant}.
In the discrepancy expression, the contact difference acquires
$2v(\widehat K_{Y/X})-v(\widehat K_{Y/X})$,
equal to the codimension change.
\end{proof}

The same calculation applies on a common smooth model factoring
through the Nash blow-up. The changes in codimension and differential
order cancel; no effectiveness assumption on the crepant boundary
is needed.

\subsection{A criterion for each divisor}

\begin{Corollary}
The pair $(\mathcal F,\Delta)$ is log canonical if and only if
\[
\vartheta_{\mathcal F,\Delta}(2\operatorname{ord}_E)
 -\vartheta_{\mathcal F,\Delta}(\operatorname{ord}_E)
 \leq\operatorname{codim}C_X(\operatorname{ord}_E)
\]
for every prime divisor $E$ over $X$.
For a fixed $r$ such that $r(K_{\mathcal F}+\Delta)$ is Cartier,
the equivalent test
\eqref{eq:intro-index} holds for each individual $E$.
\end{Corollary}
\begin{proof}
The first assertion is Theorem~\ref{thm:intro}; for the second,
Proposition~\ref{prop:general-profile} gives
\begin{equation*}
\operatorname{codim}C_X(q\operatorname{ord}_E)
 -\vartheta_{\mathcal F,\Delta}(q\operatorname{ord}_E)
 =qA_{\mathcal F}(E)+\delta_E.
\end{equation*}
Non-negativity of $A_{\mathcal F}(E)$ implies non-negativity
of this expression.
Conversely, $rA_{\mathcal F}(E)$ is an integer.
If $A_{\mathcal F}(E)<0$, it is at most $-1/r$; at $q=r+1$
the right-hand side is at most
$-(r+1)/r+1=-1/r<0$.
\end{proof}

To conclude log canonicity, the criterion must hold for every
divisor over $X$. In particular, it does not provide a finite-jet
test for the pair.
For the non-log discrepancy of a primitive divisor, the same calculation
at multiplicity one gives
\[
a(E;\mathcal F,\Delta)
 =\operatorname{codim}C_X(\operatorname{ord}_E)
       -\vartheta_{\mathcal F,\Delta}(\operatorname{ord}_E)-1.
\]
Thus canonical singularities can be tested by the corresponding
inequality for exceptional divisors, with the usual canonical convention.

\subsection{Minimal log discrepancies and log canonical thresholds}
\label{sec:thresholds}

For a closed subset $Z\subset X$, use the literal infimum convention
for the foliated minimal log discrepancy. Then
\begin{equation}\label{eq:mld}
\operatorname{mld}_{\mathcal F}(Z;X,\Delta)
 =\inf_{\substack{E\ {\rm over}\ X\\c_X(E)\subset Z}}
 \left(\operatorname{codim}C_X(\operatorname{ord}_E)
       +\vartheta_{\mathcal F,\Delta}(\operatorname{ord}_E)
       -\vartheta_{\mathcal F,\Delta}(2\operatorname{ord}_E)\right).
\end{equation}
Using $q\operatorname{ord}_E$ divides the expression by $q$.
Neither attainment of the infimum nor lower semicontinuity in
general follows from this identity.
For a foliated log canonical pair and a non-zero ideal
$\mathfrak a\subset\mathcal O_X$, the foliated log canonical
threshold along $Z$ is
\[
\inf_{\substack{v\ {\rm divisorial},\ c_X(v)\subset Z\\v(\mathfrak a)>0}}
\frac{\operatorname{codim}C_X(v)
       +\vartheta_{\mathcal F,\Delta}(v)
       -\vartheta_{\mathcal F,\Delta}(2v)}
     {v(\mathfrak a)}.
\]
Indeed, exponent $c$ subtracts $c\,v(\mathfrak a)$ from the discrepancy.

\section{Nash and Jacobian corrections}

Assume that $K_X$ is $\mathbb Q$-Cartier and choose $r$
divisible by the indices of $K_X$ and $K_{\mathcal F}+\Delta$.
Separating the Nash ideal from the conormal module gives
\[
\mathcal Q_r(\mathcal F,\Delta)
 =\omega_X^{[r]}\otimes\mathcal L_r^{-1}
 \ \subset(\Omega^1_{K/\mathbb C})^{\otimes r}
\]
by the generic determinant map. This is an invertible sheaf embedded
as a module of rational tensors. Let
$\rho_{\mathcal F,\Delta}(v)$ be its generic differential order
on $C_X(v)$ divided by $r$.

The Nash ideal of level $r$ is defined by
\[
\operatorname{im}\left[
  (\textstyle\bigwedge^n\Omega_X^1)^{\otimes r}
          \longrightarrow\omega_X^{[r]}\right]
 =\mathfrak n_{r,X}\otimes\omega_X^{[r]}.
\]
This is the Nash ideal of level $r$ \cite[Definition~2.9]{dFD}.

\begin{Proposition}
One has
\[
\mathcal P_r(\mathcal F,\Delta)
 =\mathfrak n_{r,X}\mathcal Q_r(\mathcal F,\Delta),
\qquad
\vartheta_{\mathcal F,\Delta}(v)
 =\frac1r v(\mathfrak n_{r,X})+\rho_{\mathcal F,\Delta}(v).
\]
Consequently
\begin{equation}\label{eq:nash-formula}
A_{\mathcal F}(v;X,\Delta)
 =\operatorname{codim}C_X(v)-\frac1r v(\mathfrak n_{r,X})
        +\rho_{\mathcal F,\Delta}(v)
        -\rho_{\mathcal F,\Delta}(2v).
\end{equation}
If $X$ is normal and locally complete intersection, then
\begin{equation}\label{eq:lci-formula}
A_{\mathcal F}(v;X,\Delta)
 =\operatorname{codim}C_X(v)-v(\operatorname{Jac}_X)
        +\rho_{\mathcal F,\Delta}(v)
        -\rho_{\mathcal F,\Delta}(2v).
\end{equation}
\end{Proposition}
\begin{proof}
Tensor the defining top-differential image of $\mathfrak n_{r,X}$
with $\mathcal L_r^{-1}$.
The generic map is the one used to define $\mathcal P_r$,
so their images agree.
The order of an ideal on $C_X(v)$ is its valuation $v$.
Taking generic orders proves the second equality.
Since $(2v)(\mathfrak n_{r,X})=2v(\mathfrak n_{r,X})$,
Theorem~\ref{thm:intro} gives \eqref{eq:nash-formula}.

For a locally complete intersection, $\omega_X$ is invertible
and the image of $\bigwedge^n\Omega_X^1$ in $\omega_X$
is $\operatorname{Jac}_X\omega_X$.
Hence $\mathfrak n_{r,X}=\operatorname{Jac}_X^{\,r}$.
This is also the identity underlying the equality of ordinary and
Mather--Jacobian discrepancies on an lci variety
\cite[Corollary~3.5]{dFD}.
Substitution proves \eqref{eq:lci-formula}.
\end{proof}

Outside the lci case, the Nash and Jacobian corrections differ by
the lci-defect contribution \cite[Proposition~3.4]{dFD}.
The module $\mathcal P_r$ requires neither $K_X$ to be
$\mathbb Q$-Cartier nor involutivity: the same calculation applies
to saturated co-rank one distributions.

\section{Simple foliations and tangential arcs}
\label{sec:tangential}

The generic arc of a maximal divisorial set need not be tangent
to the foliation. At a simple singularity on a threefold, however,
we show that an arc whose defining one-form pulls back to zero
must lie on a separatrix. Throughout this section, $X$ is a
normal projective threefold, $\Delta\geq0$, and
$(X,\mathcal F,\Delta)$ is log canonical, with the discrepancy
convention of Section~\ref{sec:smooth}.

\begin{Proposition}
There is a projective birational morphism $\pi:Y\to X$, with $Y$
$\mathbb Q$-factorial and klt and $(Y,\mathcal G,\Gamma)$ F-dlt, such that
\begin{equation*}
K_{\mathcal G}+\Gamma=\pi^*(K_{\mathcal F}+\Delta),
\qquad
\Gamma=\pi^{-1}_*\Delta+\sum_{E\ \pi\text{-exceptional}}\epsilon(E)E.
\end{equation*}
Applying the reduction of singularities, we obtain a smooth model
$h:W\to Y$ on which the foliation has only simple singularities.
For the transformed foliation $\mathcal H$, define
\begin{equation}\label{eq:lc-higher-boundary}
K_{\mathcal H}+B_W=h^*(K_{\mathcal G}+\Gamma).
\end{equation}
Every $h$-exceptional $P$ satisfies $\operatorname{coeff}_P B_W\leq\epsilon(P)$,
and these crepant identities preserve discrepancies.
\end{Proposition}
\begin{proof}
The F-dlt modification theorem \cite[Theorem~1.4]{CSMMP} gives
$Y,\Gamma$ as stated and an effective exceptional $G$ with
$K_{\mathcal G}+\Gamma+G=\pi^*(K_{\mathcal F}+\Delta)$.
For an extracted $E$, log canonicity implies
$-\epsilon(E)-\operatorname{coeff}_E G=a(E;\mathcal F,\Delta)
\geq-\epsilon(E)$; hence $G=0$.
Resolve $Y$ and apply Cano's reduction \cite{Can04}. The crepant
coefficient at any further exceptional $P$ is
$-a(P;\mathcal G,\Gamma)\leq\epsilon(P)$.
\end{proof}
F-dlt pairs have transverse effective boundary of coefficients at most
one and a foliated log resolution extracting only divisors with
discrepancy greater than $-\epsilon(E)$ \cite[Definition~3.6]{CSMMP}.
Their foliations are non-dicritical \cite[Theorem~11.3]{CSMMP}.
In what follows we use the full crepant boundary $B_W$, allowing
negative coefficients.

\subsection{The simple normal forms}

Write $J_m(V)$ and $J_\infty(V)$ for jets and arcs, and
$J_\infty(V;Z)$ for arcs with origin in $Z$. Arcs at a formal germ
are continuous and centred at its closed point. Tangency means
$\gamma^*\omega=0$ for a saturated defining form.
We use the two simple normal forms of
\cite[Definition~2.8]{CSMMP}.
Put $f=x_1\cdots x_r$, $D=(f=0)$, $1\leq r\leq3$:
\begin{align*}
\omega_{\mathrm I}
 &=f\sum_{i=1}^r\lambda_i\frac{dx_i}{x_i},
 \\
\sum_i a_i\lambda_i&\ne0
 \quad(0\ne a\in\mathbb Z_{\geq0}^r),\qquad \lambda_i\ne0,
 \\
\omega_{\mathrm{II}}
 &=f\left(\sum_{i=1}^k p_i\frac{dx_i}{x_i}
       +\psi(M)\sum_{i=2}^r\lambda_i\frac{dx_i}{x_i}\right),
\qquad M=\prod_{i=1}^k x_i^{p_i}.
\end{align*}
In the second form $1\leq k\leq r$, the $p_i$ are positive integers
with no common factor, $0\ne\psi\in s\mathbb C[[s]]$, and the non-zero
$\lambda_2,\ldots,\lambda_r$ satisfy positive non-resonance.
Put $\ell=\operatorname{ord}_0\psi$. If $\psi=0$, saturation gives
the first form with fewer active coordinates. These are conditions
on the reduced model, not on the original residues on $X$.

\begin{Theorem}\label{thm:confinement}
Let $p$ be a simple point. Every formal arc centred at $p$ and
tangent to the foliation, over any extension field $K/\mathbb C$,
factors through the reduced union $D$ of the formal coordinate
separatrices. Conversely, every arc on $D$ is tangent to the foliation. Thus
\[
J_\infty^{\tan}(\widehat W_p,\mathcal H;p)_{\mathrm{red}}
 =J_\infty(D;p)_{\mathrm{red}}.
\]
The equality is compatible with changes of formal adapted coordinates.
\end{Theorem}
\begin{proof}
For an arc outside $D$, write $x_i=t^{a_i}u_i$, with $a_i>0$,
$u_i(0)\ne0$, and $A=\sum a_i$. In the first form the coefficient
of $t^{A-1}dt$ is
$(\sum\lambda_i a_i)\prod u_i(0)\ne0$.
Additional regular terms inside the logarithmic parentheses have
order at least $A$.
In the second form $P=\sum_{i\leq k}p_i a_i>0$ and
$\operatorname{ord}_t\psi(M)=\ell P\geq1$.
The first summand contributes $P\prod u_i(0)t^{A-1}dt$;
the second has order at least $A$. Both contradict tangency.

Conversely, $f(\gamma)=0$ in $K[[t]]$ forces some $x_{i_0}(\gamma)=0$.
Writing $\omega=\sum c_i(f/x_i)dx_i$, every term with $i\ne i_0$
has zero coefficient and the remaining term has zero differential.
Equality over all extension fields identifies the radicals of the
centred arc equations. On overlaps both tangency and the full
separatrix union are intrinsic.
\end{proof}

Notice that in the second form the leading coefficient is determined
by the positive integers $p_i$, independently of the non-resonance of
the $\lambda_i$.
For arcs with varying centres, the normal form is taken at the
relevant centre, within its formal domain.

\begin{Proposition}
On the algebraic coordinate model of the pure logarithmic form,
a relation $a\cdot\lambda=0$,
$0\ne a\in\mathbb Z_{\geq0}^r$, gives a tangential arc centred on $D$
with generic point outside $D$.
\end{Proposition}
\begin{proof}
Take $x_i=c_i t^{a_i}$, $c_i\ne0$, and the other coordinates constant.
The logarithmic pull-back is $(a\cdot\lambda)dt/t=0$.
No coordinate is identically zero, and $a\ne0$ places the origin on $D$.
\end{proof}
With additional regular terms, the relation cancels only the leading
coefficient and need not yield a tangential arc.

\begin{Theorem}\label{thm:resonance-dicritical}
Let $D=(x_1\cdots x_r=0)\subset\mathbb A^n$, $2\leq r\leq n$,
and $\omega=(\prod_i x_i)\sum_i\lambda_i\,dx_i/x_i$,
with $\lambda_i\in\mathbb C^*$. The following are equivalent:
(i) $a\cdot\lambda\ne0$ for every
$0\ne a\in\mathbb Z_{\geq0}^r$;
(ii) every tangential formal arc starting on $D$ lies in $D$;
(iii) every toric divisor centred in $D$ is invariant.
For the divisor $E_a$ defined by a primitive weight vector $a$,
put $m=\sum_i a_i$ and $c=a\cdot\lambda$.
The saturation order is $m-1$ if $c\ne0$ and $m$ if $c=0$;
$E_a$ is dicritical exactly when $c=0$.
\end{Theorem}
\begin{proof}
For an arc not contained in $D$, write $x_i=t^{a_i}u_i(t)$,
$u_i(0)\ne0$. Its logarithmic pull-back has coefficient
$(a\cdot\lambda)dt/t$. Conversely, a relation $a\cdot\lambda=0$
gives the escaping tangential arc $x_i=t^{a_i}$, taking $x_i=1$
when $a_i=0$. On a smooth toric model, $x_i=u^{a_i}v_i$,
and at the generic point of $E_a=(u=0)$,
\[
\mu^*\omega=u^{m-1}\prod_i v_i
\left(c\,du+u\sum_i\lambda_i\,dv_i/v_i\right).
\]
For $c\ne0$ the saturated conormal is normal to $E_a$.
For $c=0$ the residual form is a non-zero form on its generic
torus: $\lambda$ annihilates $a$ but not the quotient lattice.
It is therefore tangent to $E_a$, and $E_a$ is dicritical.
The displayed factors give both saturation orders.
\end{proof}
At the deepest crossing all weights must be positive; the theorem
ranges over every stratum of $D$. Adding a regular one-form preserves
the non-resonance implications, but need not preserve their converses.

\begin{Example}
For $D=(xy=0)$ and $m\geq1$, the jet $(x,y)=(t,t^m)$ satisfies
$xy=0\bmod t^{m+1}$ although neither coordinate vanishes modulo
$t^{m+1}$. Hence $J_m(D)$ need not be the union of the branch jet
schemes. For arcs, the branch decomposition follows from the fact that
$K[[t]]$ is a domain.
\end{Example}

\subsection{Lifting arcs and tangency at regular points}

Put $\mu=\pi h$ and choose a regular open $X^\circ$ where $\mu$
is an isomorphism.
\begin{Proposition}
An arc $\gamma:\operatorname{Spec}K[[t]]\to X$ with generic point in
$X^\circ$ lifts uniquely to $\widetilde\gamma$ on $W$.
It is tangential precisely when its lift is tangential.
If its lifted origin is simple, the arc lies in the corresponding
separatrix divisor $D$. At a regular origin it lies in the formal
leaf through that point.
\end{Proposition}
\begin{proof}
The inverse on $X^\circ$ lifts the generic point, and properness
extends it uniquely. On the isomorphism locus, local saturated defining
forms differ by a unit, so their pull-backs vanish simultaneously
in $K((t))dt$.
The form on $W$ is regular, and thus vanishes in $K[[t]]dt$.
Apply Theorem~\ref{thm:confinement} or the flowbox calculation below.
\end{proof}
Arcs contained in the exceptional image require separate analysis;
their lifts need not be unique.

\begin{Proposition}
In a formal or holomorphic flowbox $(s,u,v)$ with $\mathcal H=\ker ds$,
tangential arcs satisfy $s(t)=s(0)$ and are relative arcs over $s$.
For $B=b_1(u=0)+b_2(v=0)$ in relative coordinates, the non-empty
contact cylinder $C$ of orders $(m,n)$ satisfies
\[
\operatorname{codim}_{\mathrm{rel}}C=m+n,\qquad
\operatorname{codim}_{\mathrm{rel}}C-\operatorname{ord}_C B
  =(1-b_1)m+(1-b_2)n.
\]
Absent divisors and zero coefficients are allowed.
\end{Proposition}
\begin{proof}
Frobenius gives the coordinates. In characteristic zero, $s'=0$
kills exactly the positive-degree coefficients of $s$.
The relative $j$-jet fibre has dimension $2(j+1)$; the contact
conditions impose $m+n$ independent equations for large $j$.
Subtracting the boundary order $b_1m+b_2n$ gives the formula.
\end{proof}
The formal flowbox need not come from an algebraic first integral.
When a smooth algebraic first integral exists, we have the following
comparison of discrepancies.

\begin{Proposition}
Let $f:U\to T$ be smooth from a smooth threefold to a smooth curve,
$\mathcal H=\ker df$, and $B$ a rational boundary with horizontal
support. If a divisor $F$ over $U$ dominates $T$, then, on the generic
fibre over $\mathbb C(T)$,
\[
A_{\mathcal H}(F;U,B)=A(F_\eta;U_\eta,B_\eta),\qquad
\operatorname{codim}_{\mathrm{rel}}N_q(F_\eta)
 -q\operatorname{ord}_{F_\eta}B_\eta
 =qA_{\mathcal H}(F;U,B).
\]
These numbers are non-negative when $(U,\mathcal H,B)$ is log canonical.
\end{Proposition}
\begin{proof}
On a smooth model carrying $F$, the map to $T$ is smooth at its
generic point. The transformed foliation has the relative canonical
class there and $F$ is transverse, so $\epsilon(F)=1$.
Its coefficient is the ordinary relative canonical coefficient plus
one, minus the boundary order. Restrict to the generic fibre and
apply the ordinary arc formula over $\mathbb C(T)$.
\end{proof}

\section{Foliated adjunction on separatrices}
\label{sec:systems}

Let $W_0\to X$ be a smooth threefold model and let $D_0$ be
a reduced algebraic invariant divisor with simple normal crossings
(SNC). At every point of $D_0$, the completed germ of $D_0$ is the reduced union of the
coordinate separatrices of one of the simple forms above. We also
allow regular invariant charts, with $r=1$.
Let $B=\sum a_kT_k$ be a rational divisor supported on transverse
coordinate hypersurfaces, with no invariant component, such that
$D_0+\operatorname{Supp}B$ is SNC. We study valuations of these algebraic strata. Non-algebraic
formal separatrices occur only in local statements.
Subsequent adapted models are coordinate-stratum blow-ups retaining
simple charts. Comparisons with $X$ use boundaries satisfying the crepant
identity \eqref{eq:lc-higher-boundary}.

\subsection{Arcs on a normal crossings divisor}

\begin{Proposition}
For $D=\bigcup S_i$ in an adapted chart,
$J_\infty(D;Z)_{\mathrm{red}}=\bigcup_iJ_\infty(S_i;Z)_{\mathrm{red}}$.
Over a field, its arcs are the quotient of the disjoint branch arcs
by identifications from pairwise intersections. An irreducible closed
family outside $\bigcup_{i<j}J_\infty(S_i\cap S_j)$ has a unique
generic branch.
\end{Proposition}
\begin{proof}
Since $\prod x_i(t)=0$ in a domain, some $x_i(t)=0$.
The vanishing coordinates give exactly the branch representations,
and two representations coincide on the corresponding intersection.
Pairwise identifications also generate those at a triple point.
The radical argument gives reduced equality; the generic member of
an irreducible family lies on a unique branch off the intersections.
\end{proof}
Here the quotient describes field-valued arcs; it is not a statement
about arbitrary base schemes. The arcs lying on a single branch form
$J_\infty(S_i;Z)\setminus\bigcup_{j\ne i}J_\infty(S_i\cap S_j)$.
For arcs in this locus, the orders of equations $f_j,g_k$ for
$S_j,T_k$ agree with the orders of their restrictions to $S_i$.

Write $B_{ij}=S_i\cap S_j$. The normalised system consists of
$S_i^\nu,C_{ij}^\nu$ and the two conductor incidence maps.
Triple points remain intersection points of conductor curves.
On this model the branches and their intersection curves are smooth.
\begin{Proposition}\label{prop:pushout}
The equaliser
\begin{equation*}
\mathcal O_{\mathscr S_W^{\mathrm{sn}}}
 =\left\{(f_i)\in\bigoplus_i\mathcal O_{S_i^\nu}:
         f_i|_{C_{ij}^\nu}=f_j|_{C_{ij}^\nu}\right\}
\end{equation*}
defines the reduced seminormal pushout, and
$\mathscr S_W^{\mathrm{sn}}\simeq D$ in the SNC setting.
\end{Proposition}
\begin{proof}
At a double crossing,
$R/(xy)=R/(x)\times_{R/(x,y)}R/(y)$.
At a triple crossing, compatible triples on the coordinate planes
come from $R/(xyz)$: first glue the $x$- and $y$-plane functions
modulo $xy$; their difference from the $z$-plane function lies in
$(x)\cap(y)=(xy)$ modulo $z$, and can be removed by a multiple of $xy$.
The kernel is $(x)\cap(y)\cap(z)=(xyz)$.
These square-free crossing rings are seminormal and the fibre products
have the pushout universal property. The calculations commute with
smooth parameters and glue on overlaps.
\end{proof}

\subsection{Foliated adjunction on invariant components}

\begin{Proposition}
Invariant-divisor adjunction gives
\begin{equation*}
(K_{\mathcal H}+B)|_{S_i^\nu}\sim_{\mathbb Q}K_{S_i^\nu}+\Theta_i,
\qquad
\Theta_i=L_i+R_i,\qquad
L_i=\sum_{j\ne i}B_{ij}+\sum_k a_k(T_k|_{S_i}).
\end{equation*}
For type I, $R_i=0$. For type II,
\begin{equation}\label{eq:excess-divisor}
R_i=
\begin{cases}
0,&i\leq k,\\
\ell\sum_{j\leq k}p_jB_{ij},&i>k.
\end{cases}
\end{equation}
For a conductor $C=C_{ij}^\nu$, ordinary adjunction gives
\begin{equation}\label{eq:conductor-adjunction}
(K_{S_i^\nu}+L_i)|_C\sim_{\mathbb Q}K_C+L_C,\qquad
L_C=(L_i-C)|_C,
\end{equation}
with the same $L_C$ from both adjacent branches.
\end{Proposition}
\begin{proof}
We compute the foliated different using the restricted Pfaff
field \cite{CSAdj,SS}.
If $\omega=\sum c_i(f/x_i)dx_i$, its bivector on $S_i$ is a
generator of $\bigwedge^2T_{S_i}$ multiplied by
$c_i|_{S_i}\prod_{j\ne i}x_j|_{S_i}$.
The different is the vanishing divisor of this coefficient.
In type I, $c_i=\lambda_i$ is a unit. In type II, put $\lambda_1=0$
and $p_i=0$ for $i>k$, so $c_i=p_i+\lambda_i\psi(M)$.
Its restriction is the non-zero constant $p_i$ for $i\leq k$;
for $i>k$, writing $\psi(s)=s^\ell u(s)$ gives the excess
$\ell\sum_{j\leq k}p_jB_{ij}$. The transverse boundary adds its
restrictions. The zero divisor of the restricted Pfaff coefficient, together
with the restriction of $B$, is therefore the foliated different
in these charts.
The divisor $L_i$ accounts for intersections with the other branches
and for the transverse boundary. Hence $R_i=\Theta_i-L_i$ is determined
by the specified branches; this equality does not define a different
on an arbitrary singular image.
On $C=(x=y=0)$, a remaining transverse trace $a(z=0)$ restricts
identically from either side. At a triple point, the remaining invariant
trace gives the reduced origin on $C$. These exhaust $L_i-C$ and
prove \eqref{eq:conductor-adjunction}.
Using an actual conductor coefficient greater than one would not give
this adjunction.
\end{proof}

For every divisorial valuation $w$ over an algebraic branch $S_i$,
the equality $\Theta_i=L_i+R_i$ gives
\begin{equation*}
 A(w;S_i,L_i)=A(w;S_i,\Theta_i)+w(R_i).
\end{equation*}
Indeed, the relative canonical divisor is the same on both sides,
and the two boundaries differ by $R_i$. This equality concerns ordinary discrepancies on the branch.
The ambient divisors realising these discrepancies are constructed
in Theorem~\ref{thm:regular-leaf-lift}.
We shall use the ordinary pairs on the normalised components and
their intersections, denoted by
\begin{equation*}
\mathcal S_0=\{(S_i^\nu,L_i)\}_i
              \ \cup\ \{(C_{ij}^\nu,L_{ij})\}_{i<j}.
\end{equation*}
Write $\rho_V:V\to X$ for the induced map of each stratum.
Conductors occur once, with both incidence maps.
In type II these pairs use $L_i$ rather than $\Theta_i$; in type I
the two boundaries coincide. A valuation of a branch whose centre
lies on the conductor is
still a valuation of the branch function field, not of the function
field of the conductor curve. Section~\ref{sec:mj} treats the descent of boundary ideals
to non-normal images.

\subsection{Birational pull-back of the different}

The different can be described by the morphism induced by a defining
Pfaff form. Let $S$ be a normal algebraic invariant surface
on a smooth threefold $W$. Since $S$ is Cartier, its dualising sheaf
$\omega_S$ is invertible. On the smooth locus, invariance gives a non-zero
morphism
$\omega_S\to\mathcal O_W(K_{\mathcal H})|_S$.
Normality extends this morphism uniquely across the singular set; its
zero divisor $\Theta_S^0$ is Cartier: both sheaves are invertible,
so the extension is a section of an invertible sheaf on the normal
surface. Locally, for an equation $x$ of $S$,
\[
 \omega=a\,dx+x\beta,\qquad
 \Theta_S^0=\operatorname{div}_S(a|_S).
\]
Indeed, the normal coefficient extends from $S_{\mathrm{reg}}$ to $S$.
Lift it to $W$; every coefficient of $\omega-a\,dx$ vanishes on the dense
smooth locus of $S$ and is therefore divisible by $x$.
Changing either local generator multiplies $a|_S$ by a unit.
For a rational divisor $B$ with no component $S$, put
$\Theta_S=\Theta_S^0+B|_S$. For effective $B$ this is the
adjunction different in the present setting. For a signed $B$ it denotes
the rational adjunction divisor defined by this equality, which need
not be an effective different in the usual sense. If $D$ is a reduced invariant
divisor containing $S$, define
\begin{equation}\label{eq:signed-boundary}
 I_S=(D-S)|_S,\qquad L_S=I_S+B|_S,\qquad
 R_S=\Theta_S^0-I_S.
\end{equation}
A component $T$ of $D-S$ occurring in $B$ with coefficient $b_T$
contributes trace coefficient $1+b_T$ in the SNC case. Restriction requires
$\operatorname{coeff}_S B=0$; the initial conductor rule need not
apply to new exceptional branches.

We compare the adjunction divisors on a normal invariant
surface and on a normal strict transform under a birational
morphism between smooth threefolds.

\begin{Theorem}
\label{prop:crepant-comparison}
Let $\mu:W'\to W$ be a proper birational morphism between smooth
threefolds, with transformed foliation $\mathcal H'$, and let
\begin{equation*}
 K_{\mathcal H'}+B'=\mu^*(K_{\mathcal H}+B)
\end{equation*}
use the full rational boundary. Suppose that $S\subset W$ and its
strict transform $S'\subset W'$ are normal algebraic invariant
surfaces, and that neither is a component of its respective boundary.
Write $\sigma=\mu|_{S'}$. Define $\Theta,\Theta'$ by their Pfaff
morphisms and boundary restrictions. With compatible rational
canonical forms, one has an equality of divisors
\begin{equation}\label{eq:actual-diff-pullback}
 K_{S'}+\Theta'=\sigma^*(K_S+\Theta).
\end{equation}
Consequently every divisorial valuation $w$ centred on $S$ satisfies
$
 A(w;S',\Theta')=A(w;S,\Theta).
$
For the corrections in \eqref{eq:signed-boundary},
\begin{equation}\label{eq:excess-defect}
 A(w;S',L')-A(w;S,L)=w(R'-\sigma^*R).
\end{equation}
\end{Theorem}
\begin{proof}
Since $S$ and $S'$ are normal Cartier divisors on smooth threefolds,
they are Gorenstein and their singular loci have codimension at least
two. Their canonical divisors are Cartier, and the boundary
restrictions are rational Cartier divisors. It therefore suffices
to compare the two sides at the generic
points of prime curves of $S'$, where $S'$ is smooth; equality there
extends to the normal surface. Write
$J_\mu=K_{W'}-\mu^*K_W$, $\mu^*\omega=g\omega'$ and
$M_S=\mu^*S-S'$. Locally $\mu^*x=h x'$, with
$\operatorname{div}(h)=M_S$ and $\operatorname{div}(g)=T_\mu$. Modulo $x'$, the normal
coefficient satisfies $a'=(h/g)\sigma^*a$: the terms
$x'\,dh$ and $hx'\mu^*\beta$ vanish. Hence
\begin{equation}\label{eq:normal-coefficient-transform}
\Theta_{S'}^0=\sigma^*\Theta_S^0+(M_S-T_\mu)|_{S'}.
\end{equation}
The determinant sequence gives
$K_{\mathcal H'}-\mu^*K_{\mathcal H}=J_\mu-T_\mu$, so
\begin{equation*}
B'=\mu^*B+T_\mu-J_\mu.
\end{equation*}
Adding its restriction gives
$\Theta'=\sigma^*\Theta+(M_S-J_\mu)|_{S'}$.
Adjunction for normal Cartier surfaces gives
$K_{S'}-\sigma^*K_S=(J_\mu-M_S)|_{S'}$.
Their sum proves \eqref{eq:actual-diff-pullback}.
These are divisor identities obtained from the Pfaff coefficient and
compatible canonical sections; changes of local frames are units.
All terms are rational Cartier divisors. Pulling to a common smooth
surface model proves discrepancy invariance.
Finally, subtracting $R'$ gives
\begin{equation*}
K_{S'}+L'=\sigma^*(K_S+L)+\sigma^*R-R',
\end{equation*}
and therefore \eqref{eq:excess-defect}.
\end{proof}

\begin{Proposition}\label{prop:excess-transform}
In the preceding theorem, let $D,D'$ be reduced invariant divisors
containing $S,S'$, with $\mu_*D'=D$ and no additional non-exceptional
components. Incidence restrictions retain their intersection
multiplicities; no normal crossings condition is required. Put
$Q_\mu=\mu^*D-D'-T_\mu$ and $D_\sigma=Q_\mu|_{S'}$.
Then
\begin{equation}\label{eq:excess-law}
 R'=\sigma^*R+D_\sigma,\qquad
 A(w;S',L')-A(w;S,L)=w(D_\sigma).
\end{equation}
For an exceptional prime $E$, write
$m_E=\operatorname{ord}_E D$, $s_E=\operatorname{coeff}_E T_\mu$
and $d_E=\operatorname{coeff}_E D'\in\{0,1\}$. Thus
\[
 Q_\mu=\sum_{E\ \mathrm{exceptional}}(m_E-d_E-s_E)E.
\]
Equality of corrected discrepancies for all divisorial $w$ is
equivalent to $D_\sigma=0$ as a divisor on $S'$.
\end{Proposition}
\begin{proof}
Subtract the two incidence divisors:
\[
 I_{S'}-\sigma^*I_S
   =(D'-\mu^*D+M_S)|_{S'}.
\]
Subtract this identity from
\eqref{eq:normal-coefficient-transform} to obtain the first
formula in \eqref{eq:excess-law}. The second follows from
\eqref{eq:excess-defect}. If it vanishes for every $w$, apply it
to the order of each prime curve on $S'$; every coefficient of
$D_\sigma$ is then zero. The converse follows by pull-back.
\end{proof}

For local equations $f,f'$ of $D,D'$, the identity
$\omega'/f'=(\mu^*f/(gf'))\mu^*(\omega/f)$ identifies $Q_\mu$
as the divisor comparing logarithmic conormals. Divisorial equality, not linear equivalence, controls valuation orders.
For composable $\nu:W''\to W'$ and $\mu:W'\to W$, saturation
satisfies $T_{\mu\nu}=\nu^*T_\mu+T_\nu$. Thus
\begin{equation}\label{eq:defect-composition}
Q_{\mu\nu}=\nu^*Q_\mu+Q_\nu,\qquad
D_{\sigma\tau}=\tau^*D_\sigma+D_\tau,\quad \tau=\nu|_{S''}.
\end{equation}
This follows by substitution and restriction.

\subsection{Saturation along simple separatrices}

For the simple forms, the logarithmic defining one-form has non-zero
residue along every divisor centred in the specified separatrix divisor.
This determines the saturation order without assuming that the
valuation is monomial.

\begin{Lemma}\label{lem:blowup-residues}
Let $D=(x_1\cdots x_r=0)$ be the full coordinate divisor of either
simple form. Let $E$ be a prime divisor over the chart with centre
in $D$, and put $a_j=\operatorname{ord}_E x_j$ and
$m_E=\sum_{j\le r}a_j$. Let $\eta=\omega/(x_1\cdots x_r)$ be the rational logarithmic
conormal form. On a smooth model carrying $E$, its logarithmic residue
along $E$ is non-zero; hence $E$ is invariant and the saturation
order is
\begin{equation}\label{eq:simple-saturation}
 s_E=m_E-1.
\end{equation}
The assertion holds for arbitrary divisorial valuations for which
the indicated local chart is defined, without a monomiality assumption.
\end{Lemma}
\begin{proof}
Fix a smooth model $Y$ carrying $E$ and a closed point $p$
of the centre of $E$. We check that the calculation in completed
coordinates measures the original divisorial order. Put
$A=\mathcal O_{W,p}$. Since $A$ is excellent, the faithfully flat
completion $A\to\widehat A$ is a regular morphism; its base change
to the finite-type $A$-scheme $Y\times_W\Spec A$ is regular
\cite[Tags~07QS and~07C1]{Stacks}. Choose a prime divisor
above $E$ in this base change and localise at its generic point.
The resulting regular local ring $R'$ is a DVR and the induced map
$\mathcal O_{Y,\eta_E}\to R'$ is flat and unramified in codimension
one: a uniformiser $t$ of the former remains a uniformiser of the
latter. In particular, the order of every algebraic function is
unchanged. The same holds for the common divisorial factor of an
algebraic one-form: in regular local coordinates this factor is the
minimum of the coefficient orders, and the pull-back of the
coordinate differentials is related by a matrix of units with
invertible determinant. Formal changes of adapted coordinates have
the same property, using continuous differentials on $\widehat A$.
Finally, equations of the algebraic components of $D$ and the
corresponding formal coordinate equations $x_j$ differ by units.
Consequently $\operatorname{ord}_E D=\sum_j a_j$.

The calculation can therefore be made at this DVR. No assumption
that $E$ is monomial in the formal coordinates is needed.

Write $x_j=t^{a_j}u_j$ in $R'$, with $u_j$ a unit, and put
$f=x_1\cdots x_r$. Since the centre of $E$ lies in $D$, the
non-negative integers $a_j$ are not all zero. For any $x_j$,
\[
 \frac{dx_j}{x_j}=a_j\frac{dt}{t}+\frac{du_j}{u_j}.
\]
The second term is regular at the generic point of $E$.
Accordingly the polar part of the logarithmic form
$\eta=\omega/f$ is $\rho\,dt/t$, where $\rho$ is regular.
If $\bar\rho\ne0$ modulo $t$, then
\[
 \mu^*\omega=t^{m_E-1}
       \left((\textstyle\prod_j u_j)\rho\,dt
                    +t\,(\text{regular one-form})\right).
\]
Its common divisorial factor has order exactly $m_E-1$.
After saturation the defining form is normal to $E$, so $E$
is invariant. It remains to check the residue in the two types.

In type I, $\bar\rho=\sum_j\lambda_j a_j\ne0$ by positive
non-resonance. In type II, put
$P_E=\sum_{j\le k}p_ja_j$ and
$\Lambda_E=\sum_{j=2}^r\lambda_ja_j$.
When $P_E>0$, the factor $\psi(M)$ has positive $t$-order,
so $\bar\rho=P_E\ne0$. When $P_E=0$, all positive $a_j$ have $j>k$ and
$\Lambda_E\ne0$ by positive non-resonance. Here $M$ has order
zero. Factor the formal series as
$\psi(s)=s^\ell U(s)$ with $U(0)\ne0$. The element $U(M)$ is
a unit of $\widehat A$, hence also of $R'$. Therefore
$\psi(M)=M^\ell U(M)$ is a unit of $R'$ and
$\bar\rho=\overline{\psi(M)}\Lambda_E\ne0$.
The notation means the image in $R'$ of the already defined
formal element $\psi(M)\in\widehat A$; it does not mean evaluating
an arbitrary formal series at a $t$-adic unit. This also covers valuations
that are not monomial in the chosen coordinates.

The common factor and invariance have thus been established after
the regular faithfully flat base change, and the same orders and
invariance descend to the original generic divisor $E$.
\end{proof}

The comparison uses the reduced union of all coordinate
separatrices at each point. Its birational consequences also hold
on models where the transformed foliation need no longer be simple.

\begin{Theorem}\label{thm:ordinary-pair}
Let $W$ be a smooth algebraic threefold with a co-rank one foliation
$\mathcal H$, and let $D$ be a reduced algebraic SNC invariant
divisor whose completed germ at each point of $D$ is the full
coordinate separatrix divisor of a simple normal form for $\mathcal H$, including
regular invariant charts. Let $B$ be any rational divisor.
Every prime divisor $E$
over $W$ with centre in $D$ is invariant and satisfies
\begin{equation*}
A_{\mathcal H}(E;W,B)=A(E;W,D+B).
\end{equation*}
Consequently, for every non-empty closed $Z\subset D$,
\[
\mld_{\mathcal H}(Z;W,B)=\mld(Z;W,D+B).
\]
The discrepancy function over $D$ is therefore the same for any two
simple foliations with the same reduced separatrix divisor and boundary,
irrespective of their residues or of the divisors $R_i$.
\end{Theorem}
\begin{proof}
Lemma~\ref{lem:blowup-residues} gives invariance and saturation
order $s_E=\ord_ED-1$. Thus
$A_{\mathcal H}(E)=A_W(E)-1-s_E-\ord_EB
=A_W(E)-\ord_E(D+B)$, including divisors on $W$.
Taking the infimum over divisors with centre in $Z$ proves the second
assertion. Neither
calculation requires effectiveness of $B$ or monomiality of $E$.
\end{proof}

\begin{Corollary}\label{cor:lsc-effective}
In Theorem~\ref{thm:ordinary-pair}, suppose $D+B\geq0$. Then
$p\mapsto\mld_{\mathcal H}(\{p\};W,B)$ is lower semicontinuous
on the closed points of $D$.
If, moreover, $B=\sum_j b_jT_j$ has no component in $D$,
$D+\Supp B$ is SNC, and $0\leq b_j\leq1$, then for every
closed point $p\in D$,
\[
\mld_{\mathcal H}(\{p\};W,B)
 =3-\#\{D_i:p\in D_i\}-\sum_{T_j\ni p}b_j.
\]
In this SNC case, the log canonical centres with centre in
$D$ are precisely the irreducible components of non-empty
intersections of coefficient-one components of $D+B$ that
are contained in $D$.
\end{Corollary}
\begin{proof}
Every valuation centred at $p$ is centred in $D$, so
Theorem~\ref{thm:ordinary-pair} identifies the infimum with
$\mld(p;W,D+B)$. Lower semicontinuity follows from
\cite[Theorem~0.3]{EMY}, since $W$ is smooth and $D+B$ is
effective. For the explicit formula, choose parameters
$u_1,u_2,u_3$ at $p$ and write
$(D+B)_p=\sum_{i=1}^3c_i(u_i=0)$, where $c_i\in[0,1]$
and absent components have coefficient zero. For a prime
divisor $F$ centred at $p$, the standard regular-parameter
estimate gives
$A_W(F)\geq\sum_i\ord_Fu_i$; hence
$A(F;W,D+B)\geq\sum_i(1-c_i)\ord_Fu_i
\geq3-\sum_i c_i$. The exceptional divisor of the blow-up
of $p$ attains equality. The same calculation identifies the
log canonical centres, and the discrepancy comparison
transfers them to $\mathcal H$.
\end{proof}

\begin{Theorem}
\label{thm:excess-independence}
Let $(W,\mathcal H,D)$ be as in Theorem~\ref{thm:ordinary-pair}.
Let $\mu:W'\to W$ be a proper birational morphism, with
$W'$ smooth, whose exceptional divisors have centres in $D$,
and put $D'=(\mu^{-1}D)_{\mathrm{red}}$.
Let $B$ be a rational divisor and define $B'$ by
$K_{\mathcal H'}+B'=\mu^*(K_{\mathcal H}+B)$.
Then every component of $D'$ is invariant, and
\begin{equation}\label{eq:ordinary-crepant-transform}
 K_{W'}+D'+B'=\mu^*(K_W+D+B).
\end{equation}
For every prime divisor $E$ over $W'$ with centre in $D'$,
\[
 A_{\mathcal H'}(E;W',B')=A(E;W',D'+B').
\]

Suppose in addition that $S$ is a component of $D$,
$\operatorname{coeff}_SB=0$, and its strict transform $S'$ is
normal. With $\sigma=\mu|_{S'}$, one has
\begin{equation}\label{eq:independence}
 R_{S'}=\sigma^*R_S,\qquad
 K_{S'}+L_{S'}=\sigma^*(K_S+L_S).
\end{equation}
Thus the discrepancies of the branch pairs agree for all
divisorial valuations. Neither normal crossings of $D'$ nor smoothness of all its
components is required.
On an original type-II branch the total transform is explicitly
\begin{equation}\label{eq:weak-total-transform}
 R_{S_i'}=
 \begin{cases}
 0,&i\le k,\\
 \ell\,\sigma^*\bigl(\sum_{j\le k}p_jB_{ij}\bigr),&i>k.
 \end{cases}
\end{equation}
\end{Theorem}
\begin{proof}
Every exceptional divisor satisfies
\eqref{eq:simple-saturation} in a chart at its centre.
It is invariant, so its coefficient in $D'$ is one and
$m_E-d_E-s_E=0$. Strict transforms of the old components are
also invariant. Therefore $Q_\mu=0$ as an ambient divisor.
Equivalently, $T_\mu=\mu^*D-D'$. Since
$K_{W'}-\mu^*K_W=J_\mu$ and
$B'=\mu^*B+T_\mu-J_\mu$, this proves
\eqref{eq:ordinary-crepant-transform}.
Let $F$ be any prime divisor over $W'$ with centre in $D'$, and
regard it also as a valuation over $W$. Its centre on $W$ is
contained in $D$, since $\mu(D')\subseteq D$. In particular, $F$
is invariant by Theorem~\ref{thm:ordinary-pair}. Birational
invariance of foliated discrepancies, together with
\eqref{eq:ordinary-crepant-transform}, gives the chain
\[
\begin{aligned}
 A_{\mathcal H'}(F;W',B')
 &=A_{\mathcal H}(F;W,B)\\
 &=A(F;W,D+B)\\
 &=A(F;W',D'+B').
\end{aligned}
\]
This argument takes place on a common model carrying $F$;
it does not require a simple normal form on $W'$ or an SNC
structure for $D'$.

Since $Q_\mu=0$, Proposition~\ref{prop:excess-transform}
proves $R_{S'}=\sigma^*R_S$. Theorem~\ref{prop:crepant-comparison}
then gives the second equality in \eqref{eq:independence}.
For a divisor $w$ over the normal surface $S'$, these identities
are pulled to a common resolution of $S$ and $S'$, giving the
same discrepancies on both. Formula~\eqref{eq:weak-total-transform}
follows from \eqref{eq:excess-divisor} by Cartier pull-back.
\end{proof}

\begin{Corollary}
For the normal strict transforms in
Theorem~\ref{thm:excess-independence}, the discrepancies of both
$(S,\Theta)$ and $(S,L)$ are preserved under the birational map,
as are their infima with centre conditions imposed on the original
branch. The comparison is compatible with composition and common
smooth refinements.
\end{Corollary}
\begin{proof}
Both branch pairs are crepant over the original pair. On a
common surface resolution, the coefficients of the pulled-back
log canonical divisors agree, so the discrepancies agree for
each divisorial valuation. Moreover, for every such valuation
$w$, the containment $c_{S'}(w)\subseteq\sigma^{-1}Z$ is
equivalent to $c_S(w)\subseteq Z$. Taking infima therefore
preserves the centre condition without requiring the boundary
on $S'$ to be effective. Composition follows from
\eqref{eq:defect-composition}.
\end{proof}

Ordinary inversion of adjunction applies to the pair $(S,L)$.

\begin{Corollary}

Let $(W,\mathcal H,D)$ be as in
Theorem~\ref{thm:excess-independence}, let $S$ be a component of
$D$, and let $B$ be a rational divisor with no component $S$
such that $D-S+B\geq0$. For a non-empty proper closed subset $Z\subset S$,
\begin{equation*}
 \operatorname{mld}_{\mathcal H}(Z;W,B)
       =\operatorname{mld}(Z;S,(D-S)|_S+B|_S).
\end{equation*}
Thus non-negativity of all ambient foliated discrepancies with
centre in $Z$ is equivalent to log canonicity of $(S,L)$ along $Z$. The conclusion is preserved on the
normal strict transform under the modifications of that theorem.
\end{Corollary}
\begin{proof}
Theorem~\ref{thm:ordinary-pair} identifies the left side with
$\mld(Z;W,S+(D-S)+B)$. Ordinary inversion of adjunction for the smooth
Cartier divisor $S$, with effective residual boundary, identifies this
with the right side \cite[Theorem~0.1]{EMY}. Preservation on the
normal strict transform follows from
\eqref{eq:independence} by comparing the same divisorial
valuations on a common surface model. In particular, the
statement does not apply inversion of adjunction anew on the
possibly singular surface $S'$.
\end{proof}

\subsection{Coordinate blow-ups}

Let $Z=(x_j:j\in J)=0$, with $c=|J|\in\{2,3\}$ and
$m=|J\cap\{1,\ldots,r\}|\geq1$. In the chart with pivot $h\in J$,
put $x_h=t$, $x_j=tq_j$ for $j\in J\setminus\{h\}$, and
$q_j=x_j$ otherwise. For type II set $p_j=0$ for $j>k$,
$\lambda_1=0$, $P=\sum_{j\in J}p_j$, $\Lambda=\sum_{j\in J}\lambda_j$,
$G=\prod_{j\le r,\,j\ne h}q_j$ and
$M'=t^P\prod_{j\le k,\,j\ne h}q_j^{p_j}$. Direct substitution gives
\begin{align*}
\mu^*\omega&=t^{m-1}\omega',\\
\omega'&=GC_J\,dt+tG\sum_{j\le r,\,j\ne h}c'_j\frac{dq_j}{q_j},
\\
(C_J,c'_j)&=
\begin{cases}
(\sum_{j\in J,\,j\le r}\lambda_j,\lambda_j),&\text{type I},\\
(P+\Lambda\psi(M'),p_j+\lambda_j\psi(M')),&\text{type II}.
\end{cases}\nonumber
\end{align*}
There is no further exceptional factor by
Lemma~\ref{lem:blowup-residues}, and $D'=(tG=0)$.
For type I, residues are replaced by their sums; in particular,
\begin{align*}
\lambda_x\frac{dx}{x}+\lambda_y\frac{dy}{y}
 &=(\lambda_x+\lambda_y)\frac{dx}{x}+\lambda_y\frac{du}{u},
 &&y=xu,\\
\sum\lambda_j\frac{dx_j}{x_j}
 &=(\lambda_x+\lambda_y+\lambda_z)\frac{dx}{x}
       +\lambda_y\frac{du}{u}+\lambda_z\frac{dv}{v},
 &&(y,z)=(xu,xv).
\end{align*}
Positive non-resonance persists in type I. For type II, at a coordinate
corner choose $p'_a>0$ among the transformed coefficients. Division by
$1+(\lambda'_a/p'_a)\psi(M')$ replaces the residual coefficients by
$\lambda''_j=\lambda'_j-(\lambda'_a/p'_a)p'_j$ and $\psi$ by
$\psi/(1+(\lambda'_a/p'_a)\psi)$. If the remaining coefficients are
non-zero and positively non-resonant, the chart is again simple.
The transport theorem does not require this additional test.

Writing $\beta=\ord_E B$, the full crepant boundary is
\begin{equation*}
B'=\mu^*B+(m-c)E
   =\mu^{-1}_*B+(\beta+m-c)E.
\end{equation*}
For an old branch $S$, put $e_S=\ord_E S$ and $F=E|_{S'}$.
The general transformation formulae give
\begin{align*}
\Theta^0_{S'}&=\sigma^*\Theta^0_S+(e_S-m+1)F,&
I_{S'}&=\sigma^*I_S+(e_S-m+1)F,\\
\Theta_{S'}&=\sigma^*\Theta_S-(c-1-e_S)F,&
L_{S'}&=\sigma^*L_S-(c-1-e_S)F,\quad R_{S'}=\sigma^*R_S.
\end{align*}
If $S$ contains $Z$, then $e_S=1$; if it meets a curve centre
transversely, then $e_S=0$ and the restriction is a surface point
blow-up. A disjoint branch has $F=0$.
On the new branch $E$, the intrinsic different is
$(D'-E)|_E+\operatorname{div}_E(C_J|_E)$: its excess is zero in type I and when
$P>0$, and is $\ell\operatorname{div}_E(M'|_E)$ when $P=0$ in type II.
Restriction of the full boundary to $E$ requires
$\beta+m-c=0$.

\begin{Lemma}
A smooth surface point blow-up on total boundary coefficient $b$
has crepant exceptional coefficient $b-1$ and log discrepancy $2-b$.
\end{Lemma}
\begin{proof}
Use $K_{S'}=\sigma^*K_S+F$ and
$\sigma^*\Theta=\sigma^{-1}_*\Theta+bF$.
\end{proof}

\begin{Proposition}
Coordinate blow-ups induce morphisms of the seminormal systems,
birational on strict branches, compatible with incidence and composition.
\end{Proposition}
\begin{proof}
Restrict $\mu$ to $D'\to D$ and use
Proposition~\ref{prop:pushout}. An exceptional branch over a
curve $C$ is $\mathbb P(N_{C/W})\to C$, and one over a point maps
to that point. These restrictions agree on all intersections.
\end{proof}

\begin{Example}
For $\omega=y\,dx-x^2dy$, $D=(xy=0)$, $S=(y=0)$ and $B=0$,
put $C=(x=0)|_S$. Then $\Theta=2C$, $L=R=C$.
Blowing up the origin has saturation order one and $B'=-E$.
In the chart $(x,y,z)=(ta,tb,t)$,
$\mu^*\omega=t\{ab(1-ta)dt+tb\,da-t^2a^2db\}$.
On $S'=(b=0)$ the normal coefficient $-t^2a^2$ records both the
strict curve $\widetilde C$ and $F=E|_{S'}$. Thus
\[
\Theta^0_{S'}=2\widetilde C+2F,\quad I_{S'}=\widetilde C+F,\quad
\Theta'=2\widetilde C+F,\quad L'=\widetilde C,\quad
R'=\widetilde C+F=\sigma^*C.
\]
The remaining chart, $(x,z)=(yu,yv)$, has saturated form
$u(1-yu)dy+y\,du$ and contains no weak strict branch.
Omitting $B'=-E$ would destroy crepancy on $S'$.
\end{Example}

\begin{Example}

For $\omega=y\,dx-x^2dy$, blow up $(y=x-z^2=0)$.
In $x=t+z^2$, $y=tv$, the saturated pull-back is
\[
\omega'=v\{t-(t+z^2)^2\}dt-t(t+z^2)^2dv+2tvz\,dz.
\]
Here $D'=(tv(t+z^2)=0)$ and $B'=-E$.
The weak strict branch $v=0$ maps isomorphically to $S=(y=0)$,
but $E=(t=0)$ is tangent to the strong transform $(t+z^2=0)$;
thus $D'$ is not SNC. On $v=0$, put $F=(t=0)$ and $C'=(t+z^2=0)$.
The normal coefficient gives
$\Theta^0=F+2C'$, $I=F+C'$, $\Theta'=2C'$ and $L'=R'=C'$.
This verifies the transformation formula for $R$ although the total
transform is not SNC.
Blowing up $(x=z=0)$ yields
$\omega'=yu\,dt+ty\,du-t^2u^2dy$ with $x=tu$, $z=t$.
The strong branch has $R'=0$; on the weak branch
$\Theta'=2\widetilde C+F$, $L'=\widetilde C$ and
$R'=\widetilde C+F$. Near $y=y_0\ne0$, putting
$a=\log(y/y_0)$, $M=tu$ and $\widehat u=u/(1+Ma)$ gives
$\omega'=y(1+Ma)^2d(t\widehat u)$, a type-I crossing.
The strict-branch theorem still applies at $y=0$.
\end{Example}

\begin{Example}
For $\omega=y\,dx-x^2dy$, $D=(xy=0)$ and $S=(y=0)$, blow up
the smooth curve $(x=y-z^2=0)\subset D$.
In the chart $x=u$, $y=uv+z^2$, the strict branch is
$S'=(uv+z^2=0)$, a normal surface with an ordinary double point;
the other chart, $x=tw$, $y=t+z^2$, is smooth on $S'$.
Here $B'=-E$. Let $F$ be the reduced exceptional curve on $S'$
and $\widetilde C$ the strict transform of $C=(x=0)|_S$.
Then $E|_{S'}=2F$ and $\sigma^*C=\widetilde C+2F$.
Writing $s=uv+z^2$ gives $\mu^*\omega=s\,du-u^2ds$;
the second chart gives normal coefficient $-t^2w^2$. Hence
\begin{align*}
\Theta_{S'}^0&=2\widetilde C+4F,&
I_{S'}&=\widetilde C+2F,\\
\Theta_{S'}&=2\widetilde C+2F,&
L_{S'}&=\widetilde C,\qquad
R_{S'}=\widetilde C+2F=\sigma^*C.
\end{align*}
Incidence retains its Cartier intersection multiplicity: $2F$ is
Cartier although $F$ is not Cartier at the double point.
Since $F$ is not a component of $L_{S'}=\widetilde C$,
$A(F;S',L_{S'})=1$. Theorem~\ref{thm:regular-leaf-lift} gives
ambient lifts of this valuation with discrepancy $1$ and order
$2$ on $E$, whose restriction to $S'$ is $2F$.
\end{Example}

\begin{Example}
Outside the simple class take $\omega=y^2\,dx-x^2\,dy$,
$D=(xy=0)$, $S=(y=0)$ and $B=0$.
Initially $\Theta=2C$, $L=C$, $R=C$ for $C=(x=0)|_S$.
Blow up $(x=y=0)$. In the chart $y=xu$,
\[
 \mu^*\omega=x^2\{(u^2-u)\,dx-x\,du\}.
\]
The exceptional divisor is invariant but $s_E=2\ne m_E-1=1$;
thus $Q_\mu=-E$, $B'=E$, and $S'\simeq S$. On $S'=(u=0)$ the normal coefficient is
$-x$, giving $\Theta_{S'}^0=C$, $I_{S'}=C$ and
$
 \Theta'=2C,\  L'=2C,\  R'=0,\  D_\sigma=-C.
$
The discrepancy of $(S,L)$ at $\operatorname{ord}_C$ changes from
$0$ to $-1$, although that of $(S,\Theta)$ is unchanged. The non-zero valuation of the principal divisor $-C$ accounts
for this change.
The starting foliation is not log canonical, so the simple-class
theorem does not apply.
\end{Example}

For $\omega_\ell=y\,dx-x^{\ell+1}dy$ and transverse boundary
$b(z=0)$, $0\le b\le1$, the boundary $L=C+b(z=0)|_S$ on the weak branch is log canonical, whereas the actual different
$(\ell+1)C+b(z=0)|_S$ is not for $\ell\ge1$.
Under the point blow-up,
\[
 R'=\ell(\widetilde C+F),\quad
 L'=\widetilde C+b\widetilde T+bF,\quad
 \Theta'=(\ell+1)\widetilde C+b\widetilde T+(\ell+b)F,
\]
where $\widetilde T$ is the strict transform of $(z=0)|_S$;
the ambient crepant exceptional coefficient is $b-1$.
The correction remains non-zero.

\section{Divisorial sets on separatrix strata}

We consider divisorial valuations on the normalised components
and on their intersection strata. Throughout this section, arc
codimensions are computed on these strata, rather than in the
ambient arc space of Section~\ref{sec:module}.
Let $F$ be a prime divisor over a fixed stratum pair
$(V,B_V)\in\mathcal S_0$; on a conductor curve, $F$ is a closed point.
We include all divisorial valuations of the stratum and their multiples,
and put
\begin{equation*}
A(F;V,B_V)
 =1+\operatorname{ord}_F K_{Y/V}-\operatorname{ord}_F B_V,
\qquad A(e\operatorname{ord}_F;V,B_V)=eA(F;V,B_V).
\end{equation*}
These discrepancies are those of the ordinary pairs on the strata;
the valuations are of their function fields.
Let $N_q(F)$ be the maximal divisorial set of $q\operatorname{ord}_F$
on $V$. Its codimension is measured on that stratum; its image in the
ambient arc space may have infinite codimension.

\begin{Theorem}
For every such $F$ and $q\geq1$,
\begin{equation}\label{eq:emy}
\operatorname{codim}_{J_\infty(V)}N_q(F)
 -q\operatorname{ord}_F B_V=qA(F;V,B_V).
\end{equation}
This value is independent of a crepant representative of the fixed
pair. For $e\operatorname{ord}_F$, use $N_{qe}(F)$.
\end{Theorem}
\begin{proof}
The ordinary divisorial-set formula \cite{EMY} is
$\operatorname{codim}N_q(F)=q(1+\operatorname{ord}_F K_{Y/V})$.
Subtract the boundary order. Common crepant representatives have
the same coefficient defining $A(F;V,B_V)$; replacing $q$ by $qe$
proves the last assertion.
\end{proof}
On a branch the formula reads
$\operatorname{codim}N_q(F)-q\operatorname{ord}_F\Theta_i
 +q\operatorname{ord}_F R_i=qA(F;S_i,L_i)$.
For a singular image, the corresponding calculation also involves
the correction studied in Section~\ref{sec:mj}.

\begin{Proposition}
On a smooth branch, let $C$ be the full non-empty contact cylinder
defined by invariant trace orders $m_j$ and transverse trace orders $n_k$
along distinct coordinate curves. Then
\[
\operatorname{codim}C=\sum_{j\ne i}m_j+\sum_k n_k,\qquad
\operatorname{codim}C-\operatorname{ord}_C B_V
 =\sum_k(1-a_k)n_k.
\]
\end{Proposition}
\begin{proof}
Order $m$ in one coordinate imposes $m$ vanishing coefficients and
one non-vanishing coefficient. Distinct SNC coordinates give independent
conditions. The boundary order is $\sum m_j+\sum a_kn_k$;
subtracting it cancels the coefficient-one traces.
\end{proof}
On higher models, the crepant boundary supplies every coefficient,
including those of exceptional curves.

\subsection{Comparison with ambient discrepancies}

\begin{Proposition}
Let $Z$ be a tangent coordinate centre of codimension $c$, contained
in $m$ invariant components and meeting transverse boundary with
total coefficient $\beta$. The exceptional divisor $E$ of its
blow-up satisfies
\begin{equation}\label{eq:elementary-discrepancy}
A_{\mathcal H}(E;W,B)=c-m-\beta.
\end{equation}
On each old branch containing $Z$, this equals the ordinary discrepancy
for $L_i$ of the curve datum when $c=2$, or of the exceptional curve
of the surface point blow-up when $c=3$.
\end{Proposition}
\begin{proof}
Theorem~\ref{thm:ordinary-pair} gives $A_W(E)-\ord_E(D+B)
=c-m-\beta$. On the branch, the boundary coefficient through the
centre is $m-1+\beta$. Its curve discrepancy for $c=2$, or its
point-blow-up discrepancy for $c=3$, is again $c-m-\beta$.
\end{proof}

\begin{Proposition}
In either simple form, with the coordinate boundary $B$ fixed above,
let an invariant toroidal divisor have primitive
ambient weight vector $v$, positive on at least one invariant coordinate.
Choose a branch $S_i$ through its centre and suppose that deletion of
the $i$-th coordinate gives $\bar v=ew\ne0$, with $w$ primitive.
For the associated branch divisor $F_w$,
\begin{equation*}
A_{\mathcal H}(E_v;W,B)
 =eA(F_w;S_i,L_i)
 =e\bigl(A(F_w;S_i,\Theta_i)+\operatorname{ord}_{F_w}R_i\bigr).
\end{equation*}
The equality is a calculation for the indicated toroidal valuation
and fixed algebraic branch. It persists under crepant changes of
representatives of that \emph{same} branch pair. It computes the
discrepancy over $X$ only when the ambient boundary in this formula
is the full crepant pull-back from $X$.
\end{Proposition}
\begin{proof}
By Theorem~\ref{thm:ordinary-pair},
$A_{\mathcal H}(E_v)=\sum_{j>r}v_j-\ord_{E_v}B$.
Deleting the normal weight and using the coordinate boundary gives
$eA(F_w;S_i,L_i)$ with the same boundary order. The excess identity
uses $\Theta_i=L_i+R_i$; crepant transport preserves both
discrepancy and multiplicity.
\end{proof}

When $e=1$ and $\operatorname{ord}_{F_w}R_i=0$, this agrees with
the uncorrected actual adjunction discrepancy. In general,
both the projected multiplicity and $R_i$ enter the formula.
Theorem~\ref{thm:regular-leaf-lift} constructs ambient lifts of
arbitrary branch valuations.

\subsection{Divisorial lifting from a separatrix}

To compare the valuations of a separatrix with ambient valuations,
we first construct divisorial lifts from a normal Cartier surface.

\begin{Lemma}\label{lem:gauss-orders}
Let $S'=(x=0)$ and $F'=(y=0)$ meet transversely in a smooth
threefold, at the generic point of $F'\subset S'$. Suppose a
finite collection of regular functions has the form
$f_j=y^{m_j}a_j+xb_j$, where $a_j$ is a unit and $m_j\geq0$.
For the monomial divisorial valuation with weights $(N,1)$
on $(x,y)$, every $N>\max_jm_j$ gives
$\operatorname{ord}_{(N,1)}f_j=m_j$.
\end{Lemma}
\begin{proof}
The summands have weights $m_j$ and at least $N>m_j$;
there is no leading cancellation.
\end{proof}

\begin{Lemma}\label{lem:normal-cartier-lift}
Let $M$ be a smooth algebraic threefold, let $T\subset M$ be a
normal prime Cartier divisor, and let $\Gamma$ be a rational
divisor with no component $T$. For every prime divisorial
valuation $w=\ord_F$ over $T$, there are prime divisorial
valuations $v$ over $M$ such that
\[
 c_M(v)=c_T(w),\qquad
 A(v;M,T+\Gamma)=A(w;T,\Gamma|_T).
\]
The order of each component of $\Gamma$ is the order of its
restriction under $w$. The same valuations can preserve the
orders of any finite collection of coherent ideals whose
restrictions to $T$ are non-zero, with $v(T)$ arbitrarily large.
The valuations can be realised on smooth models where their
divisors meet the strict transform of $T$ transversely along a
curve representing $F$.
\end{Lemma}
\begin{proof}
Choose an embedded resolution of $(M,T)$, so that the strict
transform of $T$ is smooth and its union with the exceptional
divisors has simple normal crossings. A divisorial valuation of
the function field of a smooth surface is represented by a prime
curve on a surface obtained by finitely many point blow-ups.
Perform the necessary blow-ups on the strict surface and at the
same time blow up those points in the ambient threefold. Blowing
up a point on an SNC divisor preserves the SNC property. We may
therefore assume that $w$ is represented by a smooth curve $F$
on the strict surface and that the total divisor is SNC.

Blow up $F$ in the ambient threefold, and denote the resulting
composite by $h:M_0\to M$. Let $T_0$ be the strict surface and
$E_0$ the last exceptional divisor. As $F$ is Cartier on the
smooth surface, the restriction of this last blow-up to $T_0$
is an isomorphism. Moreover $E_0|_{T_0}=F$, and $T_0$ and
$E_0$ are transverse at the generic point of $F$. If an older
exceptional divisor contained $F$ on the preceding surface,
the last blow-up separates its strict transform from $T_0$
at that generic point. No other exceptional divisor therefore
contains the generic point of $F$ in $T_0$.

At the generic point of $F$, choose parameters $x,y$ on
$M_0$ such that $T_0=(x=0)$ and $E_0=(y=0)$. The completed
local ring at this generic point has a regular system of
parameters $x,y$ over the residue field of $F$. If $f$ is a
regular function and $f|_{T_0}\ne0$, then
\[
 f=y^m a+xb,\qquad m=w(f|_T),
\]
where $a$ is a unit at the generic point of $F$. The primitive
monomial valuation $v_N(x)=N$, $v_N(y)=1$ consequently satisfies
$v_N(f)=m$ whenever $N>m$. This remains true if $f$ involves
cancellation in the original coordinates: $m$ is computed on
the chosen model representing $w$, not from the orders of the
original coordinate functions.

Choose $N$ larger than the finitely many orders required for
the local equations of the components of $\Gamma$ and the
non-zero restrictions of generators of the prescribed ideals.
A generator whose restriction is zero is divisible by $x$, so
its order is at least $N$. Taking minima over the generators
preserves the order of each ideal. The same calculation, applied
to local equations of boundary components, preserves their
orders with the prescribed rational coefficients.
The valuation $v_N$ is divisorial and primitive since
$\gcd(N,1)=1$. It is realised algebraically by successive
blow-ups of the intersection of the strict surface with the
newest exceptional divisor. At the generic point of $F$ the
successive exceptional valuations have weights
$(1,1),(2,1),\ldots,(N,1)$ in $(x,y)$; their centres map to $F$.
This construction does not require $w$ to be monomial in the
original coordinates on $T$.

Choose the canonical divisors compatibly with Cartier
adjunction, and put
\[
 G=K_{M_0}+T_0-h^*(K_M+T),\qquad
 \sigma=h|_{T_0}:T_0\longrightarrow T.
\]
Since $T$ is a Cartier divisor on a smooth threefold, it is
Gorenstein.
Adjunction gives an equality of divisors
\[
 G|_{T_0}=K_{T_0}-\sigma^*K_T.
\]
At the generic point of $F$, the only exceptional divisor on
$M_0$ meeting $T_0$ is $E_0$. Since $E_0|_{T_0}=F$,
\[
 \operatorname{coeff}_{E_0}G
   =\operatorname{coeff}_F(K_{T_0}-\sigma^*K_T)
   =A(w;T)-1.
\]
Since the relative canonical coefficient of a monomial
valuation with primitive weights $(N,1)$ is $N$, its ordinary
log discrepancy on the smooth threefold is $N+1$.
Subtracting the weight $v_N(T_0)=N$ gives
$A(v_N;M_0,T_0)=1$. The only component of $G$ containing
the centre of $v_N$ is $E_0$, with weight one. Thus
\[
 A(v_N;M,T)=A(v_N;M_0,T_0)+v_N(G)
   =1+(A(w;T)-1)=A(w;T).
\]
Subtracting the preserved orders of $\Gamma$ proves the
identity with boundary. All $v_N$ have centre $c_T(w)$ on
$M$. Finally, near the generic point of $F$, we may write
$h^*T=T_0+e_0E_0+\cdots$ with the omitted components not
containing that point. Thus $v_N(T)=N+e_0$, which is
unbounded. No smoothness assumption on the original surface $T$
is used.
\end{proof}

\begin{Theorem}
\label{thm:regular-leaf-lift}
Let $(W,\mathcal H,D)$ be as in
Theorem~\ref{thm:ordinary-pair}, and let $S$ be a smooth algebraic
component of $D$. Let $B$ be a rational divisor with no component $S$,
and put $L=((D-S)+B)|_S$. For every prime divisorial valuation
$w=\ord_F$ over $S$ there is a smooth proper birational
model $\pi:W'\to W$ and an invariant prime divisor $E\subset W'$
such that $E\cap S'$ generically represents $F$ on the strict transform
$S'$, $c_W(E)=c_S(F)$, and
\begin{equation*}
 A_{\mathcal H}(E;W,B)=A(w;S,L).
\end{equation*}
The order of each component of $D-S$ and of $B$ agrees with the
order of its restriction under $w$. More generally, one can
simultaneously preserve the orders of any finite collection of
coherent ideals on $W$ whose restrictions to $S$ are non-zero.
The lifts can be chosen with arbitrarily large order along $S$.
In type II, the right-hand
side is $A(w;S,\Theta)+w(R)$. The statement includes regular leaves,
for which $D=S$.

The same conclusions hold on the models of
Theorem~\ref{thm:excess-independence}: if $S'$ is a normal strict
branch, every prime divisorial valuation $w$ over $S'$ has
invariant ambient lifts over $W'$ satisfying
\[
 A_{\mathcal H'}(E;W',B')=A(w;S',L_{S'}).
\]
The centres agree on $W'$, the orders of finitely many ideals
with non-zero restriction to $S'$ can be preserved, and
$\ord_E S'$ can be arbitrarily large.
\end{Theorem}
\begin{proof}
Apply Lemma~\ref{lem:normal-cartier-lift} with
$(M,T,\Gamma)=(W,S,D-S+B)$. Theorem~\ref{thm:ordinary-pair}
identifies the resulting ordinary discrepancy with the foliated
one and gives invariance. Include the equations of all boundary
components among the finite collection of ideals.
The type-II formula follows from $\Theta=L+R$.

On a higher model, apply the lemma to
$(W',S',D'-S'+B')$. The restricted boundary is $L_{S'}$ by
\eqref{eq:signed-boundary}; the ambient discrepancy comparison
follows from Theorem~\ref{thm:excess-independence} and
\eqref{eq:ordinary-crepant-transform}. This also proves the
assertions when the strict transform of the
surface is singular and the total divisor is not SNC.
\end{proof}

\begin{Corollary}
In either simple normal form, choose $1\leq i\leq r$ and a prime
divisorial valuation $w$ of
$\operatorname{Frac}\mathbb C[[x_j:j\ne i]]$ centred at the closed
point of $S_i$. Let $f_h$ be local equations for the components of a
rational boundary $B$ not containing $S_i$, and set
$L_i=(D-S_i+B)|_{S_i}$. For a positive integer
$N>\max_h w(f_h|_{S_i})$, with maximum zero when $B=0$, put
\begin{equation*}
 v_N\!\left(\sum_{a\geq0}x_i^ag_a\right)
   =\min_{a\geq0}\{aN+w(g_a)\},
 \qquad g_a\in\mathbb C[[x_j:j\ne i]].
\end{equation*}
Then $v_N$ is primitive and divisorial, its divisor is invariant, and
$A_{\mathcal H}(v_N;\widehat W_p,B)=A(w;S_i,L_i)$.
\end{Corollary}
\begin{proof}
Extract $w$ and let $t$ be a uniformiser of its divisor. On the product
with the formal $x_i$-line, $v_N$ has coprime weights $(1,N)$ on
$(t,x_i)$ and therefore
\[
 A_{\widehat W_p}(v_N)=A_{S_i}(w)+N,\qquad
 v_N(D)=N+\sum_{\substack{j\leq r\\j\ne i}}w(x_j).
\]
The bound on $N$ gives $v_N(f_h)=w(f_h|_{S_i})$.
Lemma~\ref{lem:blowup-residues} gives invariance and saturation
order $v_N(D)-1$. Substitution proves the discrepancy equality.
\end{proof}

\begin{Example}

Take $\omega=y\,dx-x^2\,dy$ on $\mathbb A^3_{x,y,z}$ and the
weak separatrix $S_y=(y=0)$.  Blow up the origin of $S_y$ and then
the point on its exceptional curve with tangent direction $z=x$.
The second exceptional curve gives a divisorial valuation $w$ with
$
 w(x)=w(z)=1,\ w(z-x)=2,\ A_{S_y}(w)=3.
$
It is not monomial in $(x,z)$.  Here
$L_y=(x=0)|_{S_y}$ and $R_y=(x=0)|_{S_y}$.
For any $N\ge1$ the extension has $v_N(y)=N$ and
$
 A_{\mathcal H}(E_N)=3-1=2,
 \ A(w;S_y,\Theta_y)=3-2=1.
$
The difference is $w(R_y)=1$.
\end{Example}

\begin{Example}
On $W=\mathbb A^3_{u,y,z}$ take $\mathcal H=\ker(du)$,
$S=(u=0)$ and $B=(u+y^3=0)$.  Let $F$ be the exceptional
curve of the blow-up of the origin of $S$.  Then
$\operatorname{ord}_F(B|_S)=3$, whereas the exceptional divisor
$E_0$ of the ambient point blow-up has
$\operatorname{ord}_{E_0}B=1$.  The two exceptional divisors therefore have different boundary
orders.  Three successive blow-ups of the
intersection of the strict transform of $S$ with the newest exceptional
divisor give $\operatorname{ord}_E u=4$ and
$\operatorname{ord}_E y=1$; hence $\operatorname{ord}_E B=3$.
The discrepancy of this lift agrees with $A(F;S,B|_S)$.
\end{Example}

\begin{Corollary}
For a closed subset $Z\subset S$, non-negativity of all foliated
discrepancies with ambient centre in $Z$ implies log canonicity of
$(S,L)$ along $Z$.
\end{Corollary}
\begin{proof}
Every prime divisor over $S$ with centre in $Z$ has a lift with the
same centre and discrepancy by Theorem~\ref{thm:regular-leaf-lift}.
\end{proof}

The construction above concerns a fixed algebraic branch. By
contrast, the formal formula applies to valuations of the function
field of the formal branch.

\subsection{Ideals and log canonical thresholds on a separatrix}

For ideals $\mathfrak a_1,\ldots,\mathfrak a_s$ and real weights
$\mathbf t=(t_1,\ldots,t_s)\in\mathbb R_{\geq0}^s$, write
$\mathfrak a^{\mathbf t}=\prod_j\mathfrak a_j^{t_j}$ and put
\[
 A_{\mathcal H}(v;W,B;\mathfrak a^{\mathbf t})
 =A_{\mathcal H}(v;W,B)-\sum_jt_jv(\mathfrak a_j).
\]
The corresponding minimal log discrepancy is the infimum over
prime divisors with the given centre condition. Ordinary
discrepancies with ideals use the same convention.

\begin{Theorem}\label{thm:ideal-adjunction}
Let $(W,\mathcal H,D)$ be as in Theorem~\ref{thm:ordinary-pair},
let $S$ be a component of $D$, and let $B$ be a rational divisor
with no component $S$ such that $D-S+B\geq0$.
Put $L=((D-S)+B)|_S$. Suppose each
$\mathfrak a_j\mathcal O_S$ is non-zero. For every non-empty
proper closed subset $Z\subset S$ and every
$\mathbf t\in\mathbb R_{\geq0}^s$,
\begin{equation}\label{eq:ideal-adjunction}
 \mld_{\mathcal H}(Z;W,B;\mathfrak a^{\mathbf t})
 =\mld\!\left(Z;S,L;
       \prod_j(\mathfrak a_j\mathcal O_S)^{t_j}\right).
\end{equation}
Thus the weights giving log canonicity along $Z$ are the same
on the ambient foliation and on the branch. They are determined
by finitely many rational linear inequalities on a log resolution
of the branch pair and its ideals.

For a modification as in Theorem~\ref{thm:excess-independence}
with normal strict transform $S'$, the equality also holds with
ambient centre in $\mu^{-1}Z$, branch centre in $\sigma^{-1}Z$,
and the ideals pulled back from $W$. The statement on $W'$ concerns
these pulled-back ideals and the full crepant boundary $B'$; it
does not assert inversion of adjunction for arbitrary new ideals
on $W'$ or require the transformed boundary to be effective.
\end{Theorem}
\begin{proof}
Put $Y=D-S+B$, an effective rational divisor with no
component $S$, and regard the non-negative ideal weights as
part of the ordinary pair. For every divisorial valuation $v$
with $c_W(v)\subseteq Z\subset S$, Theorem~\ref{thm:ordinary-pair}
gives
\[
 A_{\mathcal H}(v;W,B;\mathfrak a^{\mathbf t})
 =A(v;W,S+Y;\mathfrak a^{\mathbf t}).
\]
Thus the two ambient infima are equal. Since $W$ and $S$ are
smooth, $S$ is a normal Cartier divisor, and no
$\mathfrak a_j\mathcal O_S$ is zero, ordinary inversion of
adjunction \cite[Theorem~0.1]{EMY} gives
\[
 \mld(Z;W,S+Y;\mathfrak a^{\mathbf t})
 =\mld\!\left(Z;S,Y|_S;
              \prod_j(\mathfrak a_j\mathcal O_S)^{t_j}\right).
\]
Here $Y|_S=L$. The cited theorem allows non-negative real
weights on closed subschemes, so no approximation by rational
weights is needed.

To see the finite description of the log canonical region, fix
one log resolution of the support of $L$ and of all the ideals
on $S$. The finitely many coefficients on this resolution give
rational affine inequalities in $\mathbf t$. Non-negativity of
their log discrepancies is equivalent to log canonicity along
$Z$; further divisorial valuations introduce no additional
inequalities once the total boundary is SNC.

Finally let $F$ be a prime divisor over $W'$. The two crepant
equalities imply
\[
\begin{aligned}
 A_{\mathcal H'}(F;W',B')
   -\sum_j t_j\ord_F(\mathfrak a_j\mathcal O_{W'})
 &= A_{\mathcal H}(F;W,B)-\sum_jt_j\ord_F\mathfrak a_j,\\
 A(F;S',L_{S'};\textstyle\prod_j
          (\mathfrak a_j\mathcal O_{S'})^{t_j})
 &= A(F;S,L;\textstyle\prod_j
          (\mathfrak a_j\mathcal O_S)^{t_j})
\end{aligned}
\]
where the second equality is read for divisors $F$ over the
surface, not for ambient divisors. On either function field,
containment of the centre in the inverse image of $Z$ is
equivalent to containment of the centre downstairs in $Z$.
Taking the corresponding infima transports
\eqref{eq:ideal-adjunction} to $W'$ and $S'$. This uses neither
effectivity of $B'$ nor inversion of adjunction on $S'$ itself.
\end{proof}

\begin{Corollary}\label{cor:ideal-lsc}
Let $(W,\mathcal H,D)$ be as in Theorem~\ref{thm:ordinary-pair},
suppose $D+B\geq0$, and let $\mathfrak a_1,\ldots,\mathfrak a_s$ be
non-zero coherent ideals on $W$. For fixed weights
$\mathbf t\in\mathbb R_{\geq0}^s$, the function
\[
 p\longmapsto
 \mld_{\mathcal H}(\{p\};W,B;\mathfrak a^{\mathbf t})
\]
is lower semicontinuous on the closed points of $D$.
\end{Corollary}
\begin{proof}
For each closed point $p\in D$, Theorem~\ref{thm:ordinary-pair}
identifies the displayed infimum with
$\mld(\{p\};W,D+B;\mathfrak a^{\mathbf t})$.
Since $D+B$ is effective and the weights are non-negative,
this is the restriction to $D$ of the lower semicontinuous
function given by \cite[Theorem~0.3]{EMY}. The assertion is
made only on the original smooth model; a full crepant
transform can have negative coefficients, as the following
example shows.
\end{proof}

\begin{Example}
Let $\mathcal H=\ker du$ on $\mathbb A^3_{u,y,z}$, let $S=(u=0)$,
and blow up the origin, with exceptional divisor $E$.
The full crepant boundary is $B'=-2E$, and
$D'=S'+E$. Formula~\eqref{eq:ordinary-crepant-transform} gives
the ordinary pair $(W',S'-E)$. The regular-parameter estimate
therefore yields, on the strict branch,
\[
 \mld_{\mathcal H'}(\{p\};W',-2E)=
 \begin{cases}
 2,&p\in S'\setminus E,\\
 3,&p\in S'\cap E.
 \end{cases}
\]
This function is not lower semicontinuous on $S'$, although the
crepant foliated pair is log canonical. Here the ordinary
boundary $S'-E$ fails the effectivity hypothesis of
Corollary~\ref{cor:ideal-lsc}.
\end{Example}

\begin{Corollary}\label{cor:threshold-lifts}
Under the hypotheses of Theorem~\ref{thm:ideal-adjunction},
suppose $(W,\mathcal H,B)$ is log canonical along $Z$ and let
$\mathfrak a\mathcal O_S\ne0$. The foliated log canonical
threshold defined in Section~\ref{sec:thresholds} is
\begin{equation}\label{eq:threshold-adjunction}
 \inf_{\substack{w=\ord_F\ {\rm over}\ S,\ c_S(F)\subset Z\\
                  w(\mathfrak a\mathcal O_S)>0}}
 \frac{A(w;S,L)}{w(\mathfrak a\mathcal O_S)}.
\end{equation}
On a type-II branch its numerator is
$A(w;S,\Theta)+w(R)$.
If a prime divisor $F$ computes this finite infimum, infinitely
many invariant prime divisors over $W$ compute the ambient
threshold, with centre $c_S(F)$ and unbounded order along $S$.
\end{Corollary}
\begin{proof}
Apply \eqref{eq:ideal-adjunction} with one ideal and variable
weight. The two sets of non-negative weights giving log
canonicity have the same supremum, which is
\eqref{eq:threshold-adjunction}. If $F$ computes it,
Theorem~\ref{thm:regular-leaf-lift} preserves both numerator and
denominator while making the normal order arbitrarily large.
The resulting primitive valuations are distinct. The type-II
formula follows from $\Theta=L+R$.
\end{proof}

\subsection{Minimal discrepancies and semicontinuity}

For $Y\subset X$ closed, put $Y_V=\rho_V^{-1}Y$ and define
\begin{equation*}
\mld(Y;\mathcal S_0)
 =\inf_{\substack{(V,B_V)\in\mathcal S_0,\ F\text{ over }V\\
                         c_V(F)\subset Y_V}}A(F;V,B_V),
\end{equation*}
with the empty infimum $+\infty$.

\begin{Corollary}
For the fixed finite system, in which each conductor occurs once,
\begin{equation*}
\mld(Y;\mathcal S_0)=\min_{(V,B_V)\in\mathcal S_0}
                           \mld(Y_V;V,B_V).
\end{equation*}
It is non-negative exactly when all representing pairs are log canonical
along $Y_V$, equivalently when their divisorial cylinders have
non-negative logarithmic codimension. The relative loci satisfy
\[
\Nlc(\mathcal S_0)=\bigcup_V\rho_V\Nlc(V,B_V),\qquad
\Nklt(\mathcal S_0)=\bigcup_V\rho_V\Nklt(V,B_V).
\]
\end{Corollary}
\begin{proof}
Group the defining valuations by their stratum and apply
\eqref{eq:emy}. The locus identities group their centre images.
\end{proof}
The non-klt locus includes centres of all dimensions: a closed
point on a coefficient-one curve in a smooth surface has minimal
log discrepancy $1$ but belongs to the non-klt locus. These relative invariants differ from the ambient infimum
\eqref{eq:mld}. On a normal strict transform, their values are preserved by the
modifications of
Theorem~\ref{thm:excess-independence}, although the finite system
itself may acquire further strata.

\begin{Theorem}

Let $\mathcal S_0$ be a finite system of smooth SNC pairs $(V,B_V)$
of positive dimension, with boundary coefficients in $[0,1]$ and
proper maps $\rho_V:V\to X$. Then
$p\mapsto\mld(\{p\};\mathcal S_0)$ is lower semicontinuous on
closed points of $X$, taking $+\infty$ when the fibre of every
$\rho_V$ is empty.
There is no restriction on the dimensions of the strata or the fibres.
The function has finitely many finite values, each computed by a prime
divisor over one of the strata.
\end{Theorem}
\begin{proof}
Fix $(V,B)$, write $d=\dim V$ and $B=\sum_{i=1}^s b_iD_i$, and
put $D_I=\bigcap_{i\in I}D_i$, with $D_\varnothing=V$ and
$\dim\varnothing=-\infty$, so a term indexed by an empty
intersection is $+\infty$. We claim that
\begin{equation}\label{eq:proper-fibre-mld}
\mld(\rho^{-1}(p);V,B)
 =\min_{I\subseteq\{1,\ldots,s\}}
 \left\{\max\{1,d-\dim(\rho^{-1}(p)\cap D_I)\}
                   -\sum_{i\in I}b_i\right\}.
\end{equation}
For an irreducible proper centre $C\subset V$ of codimension $c$,
let $I(C)=\{i:C\subset D_i\}$. At its generic point, complete the
local equations of these $D_i$ to regular parameters
$u_1,\ldots,u_c$. For every divisor $F$ centred at $C$,
$A_V(F)\geq\sum_{j=1}^c\ord_Fu_j$. Since $b_i\leq1$, this gives
\[
 A(F;V,B)\geq c-\sum_{i\in I(C)}b_i.
\]
Equality is attained by blowing up the generic smooth locus of $C$
when $c\geq2$, and by $C$ itself when $c=1$. This is the usual SNC centre
calculation \cite{AmbroMLD}.

If $C\subset\rho^{-1}(p)$, the term indexed by $I(C)$ in
\eqref{eq:proper-fibre-mld} is no greater than this value.
Conversely, take a component of maximal dimension in
$\rho^{-1}(p)\cap D_I$. If it is proper in $V$, blowing up its
generic smooth locus gives a discrepancy at most the indicated term:
any additional boundary components containing it have non-negative
coefficients. The only remaining case is $I=\varnothing$ and
$\rho^{-1}(p)=V$. The indicated value is one, realised by a prime
divisor on $V$ not contained in $\Supp B$. This proves the formula, including
empty fibres and intersections.

Each $D_I\to X$ is proper, so its fibre dimension is upper
semicontinuous \cite[Tag~0D4I]{Stacks}.
Every term in \eqref{eq:proper-fibre-mld} is therefore lower
semicontinuous. Take the finite minimum over $I$ and then over the
pairs in $\mathcal S_0$. The same formula gives finitely many possible finite values,
and the construction produces a divisor attaining each finite
minimum.
\end{proof}

Effectivity is necessary. For the identity map of $\mathbb A^1$ and
$B=-\{0\}$, the minimal log discrepancy is $2$ at the origin and $1$
elsewhere, so it is not lower semicontinuous.

\section{Mather--Jacobian discrepancies on algebraic images}
\label{sec:mj}

The strata considered above are smooth, but their algebraic images
may be singular. The Jacobian contribution must therefore be separated
from the change in boundary ideals under descent.
Let $Z_X$ be the reduced union of the two-dimensional images of the
initial branches, and let $Z_X^{\mathrm{sn}}$ be its seminormalisation.
Lower-dimensional branch images and image curves are recorded separately,
with normalisations when required.
\begin{Proposition}
The reduced image $Z_X$ and its seminormalisation do not change
under subsequent adapted refinements of the initial model.
\end{Proposition}
\begin{proof}
The proper surjection $D'=(\mu^{-1}D)_{\mathrm{red}}\to D$ preserves
the reduced image and hence its two-dimensional components. Iterate
and use uniqueness of seminormalisation.
\end{proof}
This assertion concerns only the reduced image; it does not
imply descent of the adjunction divisors.

\begin{Assumption}
On each equidimensional algebraic image stratum $V$, choose ideals
$\mathfrak b_V=\prod_\ell\mathfrak b_{V,\ell}^{\,b_\ell}$,
non-zero on every component. These ideals specify the boundaries
on the images. Their valuation
orders need not agree with the orders of the boundaries $L_i$ on
the smooth branches.
\end{Assumption}
For an algebraic simple branch, Proposition~\ref{prop:image-ideal}
gives an image ideal which is unchanged under smooth crepant
refinements. Corollary~\ref{cor:two-defects} specifies when the
orders agree with those of $L_i=\Theta_i-R_i$, also in type II.
The normalisation and conductor of $Z_X^{\mathrm{sn}}$ are
intrinsic. Once the ideals on the normalised branches are fixed,
their contraction requires no resolution of the intersections.

\subsection{Jacobian orders and divisorial sets}

Choose $\rho:Y\to V$ resolving the Jacobian and boundary ideals and
factoring through Nash. For $d=\dim V$, write the top-differential
image as $\mathcal O_Y(-\widehat K_{Y/V})\otimes\omega_Y$ and
$\mathfrak j_V\mathcal O_Y=\mathcal O_Y(-J_{Y/V})$.
For a divisor $F$ over this fixed image pair, put
\begin{equation}\label{eq:mj-discrepancy}
A_{\mathrm{MJ}}(F;V,\mathfrak b_V)
 =1+\operatorname{ord}_F(\widehat K_{Y/V}-J_{Y/V})
       -\operatorname{val}_F\mathfrak b_V,
\end{equation}
where $\operatorname{val}_F\mathfrak b_V
=\sum_\ell b_\ell\operatorname{ord}_F\mathfrak b_{V,\ell}$.

\begin{Proposition}
For a fixed image pair, the discrepancy is independent of the
resolution and of the model on which the valuation is represented.
\end{Proposition}
\begin{proof}
On a common resolution $h:Y'\to Y$,
$\widehat K_{Y'/V}-J_{Y'/V}
 =K_{Y'/Y}+h^*(\widehat K_{Y/V}-J_{Y/V})$,
and the boundary divisor pulls back.
The discrepancy coefficient is therefore unchanged \cite{dFD,EI}.

\end{proof}

The Mather codimension formula \cite{dFEI,IshiiMather} gives
\begin{equation}\label{eq:mather-codim}
\operatorname{codim}N_q(F)
 =q(1+\operatorname{ord}_F\widehat K_{Y/V}).
\end{equation}
Codimension on a singular stratum is computed from sufficiently
high truncations of the non-thin divisorial set.
\begin{Theorem}
For every prime divisor $F$ over a fixed image pair and $q\geq1$,
\begin{equation}\label{eq:mj-lcodim}
\operatorname{codim}N_q(F)
 -q\operatorname{ord}_F\mathfrak j_V-q\operatorname{val}_F\mathfrak b_V
 =qA_{\mathrm{MJ}}(F;V,\mathfrak b_V).
\end{equation}
\end{Theorem}
\begin{proof}
Substitute \eqref{eq:mather-codim} and
$\operatorname{ord}_FJ_{Y/V}=\operatorname{ord}_F\mathfrak j_V$
in \eqref{eq:mj-lcodim}.
\end{proof}
For the fixed finite image system $\mathcal S_X$,
\begin{equation*}
\operatorname{mld}_{\mathrm{MJ}}(Y;\mathcal S_X)
 =\min_{(V,\mathfrak b_V)\in\mathcal S_X}
       \operatorname{mld}_{\mathrm{MJ}}(Y_V;V,\mathfrak b_V).
\end{equation*}
The value is independent of the chosen resolutions.

Mather--Jacobian discrepancies can also be expressed in terms
of tangent spaces to jet schemes. The formula uses the generic
truncation of the non-thin divisorial set $N_1(F)$, rather than an
arbitrary point of $J_m(V)$.

\begin{Proposition}
Let $\eta_{F,m}$ be the generic point of the closure of the
$m$-th truncation of $N_1(F)$, put
$\kappa_m=\mathbb C(\eta_{F,m})$, and define
\[
 \delta_{F,m}=\operatorname{trdeg}_{\mathbb C}\kappa_m+
 \dim_{\kappa_m}T_{\eta_{F,m}}J_m(V).
\]
Here $T_{\eta_{F,m}}J_m(V)$ is the fibre of the relative tangent
scheme over $\mathbb C$, namely
$\operatorname{Hom}_{\kappa_m}
(\Omega^1_{J_m(V)/\mathbb C}\otimes\kappa_m,\kappa_m)$.
For all sufficiently large $m$ (depending on $F$),
\begin{equation*}
 A_{\mathrm{MJ}}(F;V,\mathfrak b_V)
 =2d(m+1)-\delta_{F,m}-\operatorname{val}_F\mathfrak b_V.
\end{equation*}
\end{Proposition}
\begin{proof}
Set $e=\operatorname{ord}_F\mathfrak j_V$ and
$\widehat k_F=\operatorname{ord}_F\widehat K_{Y/V}$.
For sufficiently large $m$, the divisorial-cylinder codimension formula
\eqref{eq:mather-codim} yields
\[
 \operatorname{trdeg}_{\mathbb C}\kappa_m
 =d(m+1)-(\widehat k_F+1).
\]
At the generic liftable jet, the tangent-space formula of
\cite[Proposition~5.1 and Theorem~5.3]{dFD} gives
\[
 \dim_{\kappa_m}T_{\eta_{F,m}}J_m(V)=d(m+1)+e.
\]
Indeed, over the discrete valuation ring of the generic lifted arc,
a presentation of $\Omega_V$ can be diagonalised. The free
rank-$d$ summand contributes $d(m+1)$ to the tangent dimension;
for $m$ larger than the torsion exponents, the torsion summands
contribute their total length $e$. Thus
$\delta_{F,m}=2d(m+1)-(\widehat k_F+1)+e$.
Substituting into \eqref{eq:mj-discrepancy} proves the claim.
\end{proof}

\begin{Proposition}
If $V$ is normal and locally complete intersection, then
$\widehat K_{Y/V}-J_{Y/V}=K_{Y/V}$ and
$A_{\mathrm{MJ}}(F;V,\mathfrak b_V)=A(F;V,\mathfrak b_V)$.
Equality with the resolved adjunction discrepancy holds when that
pair is a crepant representative of the same image pair.
\end{Proposition}
\begin{proof}
For an lci variety the image of top differentials in the dualising
sheaf is its Jacobian ideal. Hence
$\widehat K_{Y/V}=K_{Y/V}+J_{Y/V}$ \cite[Corollary~3.5]{dFD}.
Substitute in \eqref{eq:mj-discrepancy} and use crepant invariance.
\end{proof}
Comparison with foliated adjunction also requires the
algebraic branches and their boundary data to be identified.

\begin{Proposition}\label{prop:image-ideal}
Let $\pi:S\to V$ be a proper birational map from a smooth algebraic
separatrix branch to the normalisation of its image. Suppose its
boundary $L$ is effective and $rL$ is integral for an integer $r>0$.
Then
\begin{equation*}
 \mathcal I_{V,r}=\pi_*\mathcal O_S(-rL)\subseteq\mathcal O_V
\end{equation*}
is a non-zero coherent ideal. It is unchanged if $S$ is replaced by a
smooth model $\sigma:S'\to S$ for which
$K_{S'}+L'=\sigma^*(K_S+L)$. The same holds for models admitting a common crepant smooth
refinement. For a divisor $w$ over $S$, put
\[
 g_{S,r}(w)=w(\mathcal I_{V,r})-r\,w(L).
\]
Then $g_{S,r}(w)\geq0$, and it vanishes for every such $w$ if
and only if
$\mathcal I_{V,r}\mathcal O_S=\mathcal O_S(-rL)$.
\end{Proposition}
\begin{proof}
Proper push-forward is coherent and
$\pi_*\mathcal O_S=\mathcal O_V$ by normality. Crepancy gives
$\mathcal O_{S'}(-rL')=\sigma^*\mathcal O_S(-rL)
\otimes\mathcal O_{S'}(rK_{S'/S})$.
The relative canonical divisor is effective and exceptional, so
$\sigma_*\mathcal O_{S'}(rK_{S'/S})=\mathcal O_S$ by extension
across codimension two. Projection formula, then a common refinement,
proves invariance; the crepant boundary $L'$ may have negative
exceptional coefficients. The evaluation map
$\pi^*\mathcal I_{V,r}\to\mathcal O_S(-rL)$ proves the
order inequality. Its image, tensored with $\mathcal O_S(rL)$,
is an ideal $\mathfrak g\subseteq\mathcal O_S$, with
$g_{S,r}(w)=w(\mathfrak g)$. If $\mathfrak g$ is proper, a
prime component of its zero locus, or the exceptional divisor
of the blow-up of a point when that locus is zero-dimensional, has positive
order. Thus vanishing for every divisorial $w$ is equivalent
to $\mathfrak g=\mathcal O_S$, the generation condition.
\end{proof}

\begin{Example}
Let $\pi:S=\operatorname{Bl}_0\mathbb A^2_{u,v}\to V=\mathbb A^2$
have exceptional curve $F$, and let $C'$ be the strict transform
of $C=(u=0)$. The boundary $L=C'$ satisfies
$
 K_S+L=\pi^*(K_V+C).
$
For every $r>0$,
\[
 \mathcal I_{V,r}=(u^r),\qquad
 \mathcal I_{V,r}\mathcal O_S=\mathcal O_S(-rC'-rF),\qquad
 g_{S,r}(\ord_F)=r.
\]
The exceptional factor measures the failure of generation on a
smooth crepant model, although the image ideal itself is unchanged.
\end{Example}

\begin{Theorem}\label{thm:conductor-sandwich}
Let $Z$ be a reduced scheme of finite type with finite normalisation
$\nu:\widetilde Z=\coprod_iV_i\to Z$, and let
$\mathfrak c\subset\mathcal O_{\widetilde Z}$ be the conductor.
Given non-zero coherent ideals $\mathcal I_i\subset\mathcal O_{V_i}$, put
$\mathcal I=\bigoplus_i\mathcal I_i$ and
$\mathcal J=\mathcal O_Z\cap\nu_*\mathcal I$.
On each $V_i$ write
$\mathcal J_i=\mathcal J\mathcal O_{V_i}$ and
$\mathfrak c_i=\mathfrak c\mathcal O_{V_i}$. Then
\begin{equation}\label{eq:conductor-sandwich}
 \mathfrak c_i\mathcal I_i\subseteq\mathcal J_i\subseteq\mathcal I_i.
\end{equation}
For every divisorial valuation $w$ of $\mathbb C(V_i)$,
\begin{equation}\label{eq:conductor-valuation}
 0\leq w(\mathcal J_i)-w(\mathcal I_i)\leq w(\mathfrak c_i).
\end{equation}
In particular, $\mathcal J_i=\mathcal I_i$ away from the conductor.
For an effective ideal boundary $\mathcal I_i^b$, $b\geq0$,
contracting and extending its underlying ideal changes the ordinary
Mather--Jacobian discrepancy by at most $b\,w(\mathfrak c_i)$:
\begin{align*}
 A_{\mathrm{MJ}}(w;V_i,\mathcal I_i^b)-b\,w(\mathfrak c_i)
 &\leq A_{\mathrm{MJ}}(w;V_i,\mathcal J_i^b)\\
 &\leq A_{\mathrm{MJ}}(w;V_i,\mathcal I_i^b).
\end{align*}
The bound is sharp.
\end{Theorem}
\begin{proof}
The conductor is an ideal of both $\mathcal O_Z$ and
$\nu_*\mathcal O_{\widetilde Z}$. The inclusions are
verified locally on $Z$, and extension to a normal component
preserves inclusions. Thus every section of
$\mathfrak c\mathcal I$ belongs to $\mathcal O_Z$ and
$\nu_*\mathcal I$, hence to $\mathcal J$; the opposite
inclusion in $\mathcal I$ is immediate. Restriction to a component
proves \eqref{eq:conductor-sandwich}.
Taking orders reverses inclusions and gives
\eqref{eq:conductor-valuation}. The discrepancy bounds follow
from the definition $A_{\mathrm{MJ}}(w;V_i,\mathcal I_i^b)
=A_{\mathrm{MJ}}(w;V_i)-b\,w(\mathcal I_i)$.
For sharpness, take $Z=(xy=0)\subset\mathbb A^3$ and
$\mathcal I_x=\mathcal O_{(x=0)}$,
$\mathcal I_y=(x)\subset\mathcal O_{(y=0)}$.
The contraction is $(x,y)$, whose extension to $(x=0)$ is
$(y)=\mathfrak c_x\mathcal I_x$; the order difference along
$C=(x=y=0)$ is exactly $w(\mathfrak c_x)=1$.
\end{proof}
Thus the difference between contraction followed by extension and
the original ideal is supported on the conductor.

\begin{Proposition}
In the notation of Theorem~\ref{thm:conductor-sandwich}, the
quotient $\mathcal I_i/\mathcal J_i$ is a coherent module
annihilated by $\mathfrak c_i$. Exact extension to the $i$-th
branch holds precisely when this quotient vanishes.
It holds on every branch if $\mathcal I_i\subseteq\mathfrak c_i$ for
all $i$.
For a double normal crossing $Z=S_1\cup S_2$ with
$C=S_1\cap S_2$, let $\overline{\mathcal I}_i$ be the image of
$\mathcal I_i$ in $\mathcal O_C$. Then
\[
\mathcal J_i=\mathcal I_i\ (i=1,2)
\quad\Longleftrightarrow\quad
\overline{\mathcal I}_1=\overline{\mathcal I}_2.
\]
For any divisorial valuation $w$ on $V_i$, set
$e_i(w)=w(\mathcal J_i)-w(\mathcal I_i)$. Then
\begin{equation}\label{eq:mj-correction}
 A_{\mathrm{MJ}}(w;V_i,\mathcal I_i^b)
 =A_{\mathrm{MJ}}(w;V_i,\mathcal J_i^b)+b\,e_i(w),
 \qquad 0\leq e_i(w)\leq w(\mathfrak c_i).
\end{equation}
\end{Proposition}
\begin{proof}
The sandwich in Theorem~\ref{thm:conductor-sandwich} gives
the first assertion. If $\mathcal I_i\subseteq\mathfrak c_i$,
their direct sum lies in $\mathcal O_Z$, so the extensions are exact.
For a double crossing,
$\mathcal O_Z=\mathcal O_{S_1}\times_{\mathcal O_C}\mathcal O_{S_2}$.
If the two images on $C$ agree, any local section of one branch ideal
can be paired with a section of the other having the same restriction.
These pairs generate both ideals after extension. Conversely, if both
extended ideals are recovered, their images on $C$ are generated by
the common restrictions of elements of $\mathcal J$ and so agree.
Finally, \eqref{eq:mj-correction} follows by subtracting the two
boundary orders in the definition of Mather--Jacobian discrepancy.
\end{proof}

\begin{Corollary}
Let $D=\bigcup S_i$ be the algebraic SNC separatrix divisor of
\secref{sec:systems}. Suppose $B$ is effective and choose a positive
integer $r$ such that every $rL_i$ is integral. The ideals $\mathcal I_i=\mathcal O_{S_i}(-rL_i)$ descend exactly to an
ideal on $D$: its extension to each $S_i$ is $\mathcal I_i$.
\end{Corollary}
\begin{proof}
The conductor on $S_i$ is the ideal of
$\bigcup_{j\ne i}(S_i\cap S_j)$. Each of these curves occurs in
$L_i$ with coefficient one, and $B\geq0$. Thus
$\mathcal I_i\subseteq\mathfrak c_i$, and the proposition applies.
\end{proof}
On singular images, birational transport can change boundary orders.

\begin{Corollary}\label{cor:two-defects}
Suppose each normalised image branch $V_i$ has a smooth birational
separatrix model $\pi_i:S_i\to V_i$ as in
Proposition~\ref{prop:image-ideal}, and take
$\mathcal I_i=\mathcal I_{V_i,r}$. With $\mathcal J_i$ as in
Theorem~\ref{thm:conductor-sandwich}, for every divisorial
valuation $w$ over $S_i$ one has
\begin{equation}\label{eq:two-defects}
 w(\mathcal J_i)-r\,w(L_i)
 =g_{S_i,r}(w)+e_i(w),\qquad
 g_{S_i,r}(w)\geq0,\quad 0\leq e_i(w)\leq w(\mathfrak c_i).
\end{equation}
The contracted ideal $\mathcal J$ contains every ideal on the
seminormal image whose extensions lie in the prescribed branch
ideals. Its extensions consequently have the smallest valuation
orders among such descended ideals. Moreover,
\[
 e_i(w)=0\ \text{for all divisorial }w
 \quad\Longleftrightarrow\quad
 \overline{\mathcal J_i}=\overline{\mathcal I_i}.
\]
Finitely many Rees valuations of $\mathcal J_i$ suffice for the
test, including the divisorial components of its support.
Comparison with the orders of the boundary $L_i$ is exact
when the evaluation map in Proposition~\ref{prop:image-ideal}
generates $\mathcal O_{S_i}(-rL_i)$ and these integral closures agree.
\end{Corollary}
\begin{proof}
Insert $w(\mathcal I_i)$ in \eqref{eq:two-defects} and apply
Proposition~\ref{prop:image-ideal} and
Theorem~\ref{thm:conductor-sandwich}. Every admissible
image ideal lies in $\mathcal J$. The valuative criterion for
integral closure gives the equivalence; the normalised blow-up of
$\mathcal J_i$ yields the finite Rees-valuation test.
\end{proof}
Thus replacing the boundary order $w(L_i)$ by
$r^{-1}w(\mathcal J_i)$ changes it by
$(g_{S_i,r}(w)+e_i(w))/r$; neither summand need vanish.

The intrinsic discrepancy of the image must also be taken
into account. Identify the function fields of $S_i$ and $V_i$.
For every prime divisorial valuation $w$ of this field, we obtain
\begin{equation}\label{eq:complete-mj-comparison}
\begin{aligned}
&A_{\mathrm{MJ}}(w;V_i,\mathcal J_i^{1/r})
       - A(w;S_i,L_i)\\
&\qquad = A_{\mathrm{MJ}}(w;V_i)-A(w;S_i)
       -\frac{g_{S_i,r}(w)+e_i(w)}{r}.
\end{aligned}
\end{equation}
The first difference on the right is intrinsic to the singular image;
it does not arise from the boundary ideal. Consequently, exact
generation and descent alone do not imply equality of the two discrepancies. Equality for a given
valuation holds precisely when that first difference equals
$(g_{S_i,r}(w)+e_i(w))/r$. In particular, if both ideal corrections
vanish, equality still requires
$A_{\mathrm{MJ}}(w;V_i)=A(w;S_i)$. The first difference need not have a fixed sign. The rational normal cubic cone in
Example~\ref{ex:mj-cubic-cone} illustrates this distinction.

\begin{Corollary}
With the smooth branch models in Corollary~\ref{cor:two-defects},
let $\mathcal I_{i,m}=\pi_{i*}\mathcal O_{S_i}(-mrL_i)$ for $m\geq1$.
Let $\mathcal J_{i,m}$ be the extension of their contraction to
$Z_X^{\mathrm{sn}}$. For every divisorial valuation $w$ of $V_i$,
\begin{equation*}
 \beta_i(w):=\lim_{m\to\infty}\frac{w(\mathcal I_{i,m})}{mr}
 =\lim_{m\to\infty}\frac{w(\mathcal J_{i,m})}{mr}
 \geq w(L_i).
\end{equation*}
The limits are unchanged by replacing $r$ with a positive multiple,
and by permitted smooth crepant branch refinements. In particular,
$A_{\mathrm{MJ}}(w;V_i)-\beta_i(w)$ is well defined for the
specified normalised algebraic components, without requiring the
branch ideals to descend exactly through the conductor. Moreover,
$\beta_i(w)=w(L_i)$ if and only if
$\lim_{m\to\infty}g_{S_i,mr}(w)/(mr)=0$.
\end{Corollary}
\begin{proof}
Multiplication of sections gives
$\mathcal I_{i,m}\mathcal I_{i,n}\subseteq\mathcal I_{i,m+n}$.
The orders are subadditive, so their normalised limit exists.
By Theorem~\ref{thm:conductor-sandwich},
$0\leq w(\mathcal J_{i,m})-w(\mathcal I_{i,m})
\leq w(\mathfrak c_i)$ independently of $m$.
Division by $mr$ proves equality of limits. Replacing $r$ by $kr$ selects the subsequence whose degrees are
divisible by $k$; subadditivity gives the same limit.
Model independence follows from Proposition~\ref{prop:image-ideal},
and the lower bound from the evaluation map on $S_i$.
\end{proof}

For a fixed valuation, the conductor contributes an error bounded
independently of $m$, which vanishes after division by $mr$.
Failure of generation, on the other hand, may persist in the limit.
In the blow-up example preceding Theorem~\ref{thm:conductor-sandwich},
take $r=1$ and $w=\operatorname{ord}_F$. Then
$w(\mathcal I_{V,m})=m$ while $w(L)=0$, so $\beta(w)=1$.
The two corrections therefore have different asymptotic behaviour.

\subsection{Discrepancies and thresholds on locally complete intersection images}

The preceding estimates compare boundary ideals. To compare
discrepancies, we must also account for the singularities of the image. For a
normal locally complete intersection variety, Mather--Jacobian and
ordinary discrepancies agree \cite[Corollary~3.5]{dFD}. This allows
us to compare the two pairs on a common smooth model, without
requiring effective boundaries.

\begin{Proposition}\label{prop:lci-threshold-comparison}
Let $V$ be a normal locally complete intersection variety and let
$\pi:S\to V$ be a proper birational morphism with $S$ smooth.
Fix a rational divisor $L$ on $S$, an integer $r>0$ with $rL$
integral, and non-zero coherent ideals $\mathcal J,\mathfrak a$
on $V$. Choose a smooth model $h:Y\to S$ on which
\[
 \mathcal J\mathcal O_Y=\mathcal O_Y(-G),\qquad
 \mathfrak a\mathcal O_Y=\mathcal O_Y(-A)
\]
are invertible ideals. Define the rational divisors $L_Y$ and $H$
by
\begin{equation}\label{eq:lci-threshold-divisor}
 K_Y+L_Y=h^*(K_S+L),\qquad
 H=\frac{G}{r}-K_{Y/V}-L_Y,
\end{equation}
where $K_{Y/V}=K_Y-(\pi h)^*K_V$.
For every prime divisorial valuation $w$ of $\mathbb C(V)$,
\begin{equation}\label{eq:lci-threshold-identity}
 A_{\mathrm{MJ}}(w;V,\mathcal J^{1/r})
   =A(w;S,L)-w(H).
\end{equation}
In particular, equality of the two discrepancies for all $w$
holds if and only if $H=0$.

Let $T\subset V$ be closed, and take the infima over prime divisors
$w$ with $c_V(w)\subseteq T$ and $w(\mathfrak a)>0$:
\begin{align*}
 \tau_{\mathrm{MJ}}(T)
 &=\inf_w\frac{A_{\mathrm{MJ}}(w;V,\mathcal J^{1/r})}
                    {w(\mathfrak a)},\\
 \tau_S(T)
 &=\inf_w\frac{A(w;S,L)}{w(\mathfrak a)}.
\end{align*}
If $H=0$, these infima agree. If $H\geq0$, then
$\tau_{\mathrm{MJ}}(T)\leq\tau_S(T)$; the reverse inequality
holds when $H\leq0$. More generally, if $\kappa\geq0$ and
\begin{equation}\label{eq:lci-threshold-bound}
 -\kappa A\leq H\leq\kappa A,
\end{equation}
then, whenever the two infima are finite,
\begin{equation*}
 |\tau_{\mathrm{MJ}}(T)-\tau_S(T)|\leq\kappa.
\end{equation*}
These conclusions apply, in particular, when $\mathcal J$ is the
extended contracted boundary ideal of a normalised image component.
\end{Proposition}
\begin{proof}
Since $V$ is locally complete intersection,
$A_{\mathrm{MJ}}(w;V)=A(w;V)$ for every $w$.
As $\mathcal J\mathcal O_Y$ is invertible,
$w(\mathcal J)=w(G)$. The ordinary discrepancy identities are
\begin{align*}
 A(w;V)&=A(w;Y)+w(K_{Y/V}),\\
 A(w;S,L)&=A(w;Y)-w(L_Y).
\end{align*}
Their difference gives \eqref{eq:lci-threshold-identity}.
If this difference vanishes for every divisorial $w$, apply it
to the valuations of prime divisors on $Y$ to obtain $H=0$.
The converse is immediate.

The two infima use exactly the same valuations: properness of
$\pi h$ identifies their centres on $V$, and pulling back
$\mathfrak a$ does not change its valuative orders.
Every divisor above $Y$ has non-negative order on an effective
Cartier divisor on $Y$. Thus $H\geq0$ or $H\leq0$ gives the
respective pointwise inequality between the numerators.
From \eqref{eq:lci-threshold-bound} we similarly obtain
$|w(H)|\leq\kappa w(A)=\kappa w(\mathfrak a)$.
Division by the positive denominator, followed by taking infima,
proves the asserted bound.
\end{proof}

For the image ideals of Corollary~\ref{cor:two-defects}, the
criterion can be expressed in terms of the contracted ideal. Assume
that
$\mathcal J\mathcal O_S\subseteq\mathcal O_S(-rL)$, and write
\[
 \mathfrak d=\mathcal J\mathcal O_S\otimes
              \mathcal O_S(rL)\subseteq\mathcal O_S,
 \qquad \mathfrak d\mathcal O_Y=\mathcal O_Y(-Q).
\]
Then $Q$ is effective and $G=rh^*L+Q$. Since
$K_{Y/V}+L_Y=h^*(K_{S/V}+L)$, formula
\eqref{eq:lci-threshold-divisor} becomes
\begin{equation}\label{eq:relative-defect-comparison}
 H=\frac{Q}{r}-h^*K_{S/V}.
\end{equation}
When $\mathcal J$ is obtained by contracting and extending the
branch ideals as above, one has
$w(Q)=g_{S,r}(w)+e(w)$ for each divisorial valuation $w$.
Thus equality of the discrepancy functions is equivalent to
$Q=rh^*K_{S/V}$. The ideal corrections need not vanish: they may be cancelled
by the relative canonical divisor. The bound
\eqref{eq:lci-threshold-bound} can be checked coefficient
by coefficient on $Y$. It does not assert that a divisor on $Y$
computes either threshold.

\begin{Example}\label{ex:lci-generation-cancellation}
Consider the normal quadric cone
\[
 V=\operatorname{Spec}
       \mathbb C[x,y,z]/(xy-z^2)
\]
and let $C\subset V$ be the prime curve defined by $(x,z)$.
Blow up the vertex to obtain the minimal resolution
$\pi_0:S_0\to V$ with exceptional curve $F$; write $C_0$
for the strict transform of $C$. In the chart
$y=s$, $z=st$, $x=st^2$, one has $F=(s=0)$ and $C_0=(t=0)$.
Since $V$ is a quadratic hypersurface, adjunction on the ambient
blow-up gives $K_{S_0/V}=0$ and
$\operatorname{div}_{S_0}(x)=F+2C_0$.

Blow up the point $F\cap C_0$ and denote the new exceptional
curve by $E$, and the strict transforms of $F,C_0$ by $F',C'$.
For the composite $\pi:S\to V$,
\[
 K_{S/V}=E,\qquad
 \operatorname{div}_S(x)=F'+2C'+3E.
\]
Take
\[
 L=C'+\tfrac12 F'+\tfrac12 E,\qquad
 r=2,\qquad \mathcal J=(x),\qquad
 \mathfrak a=(x,y,z).
\]
Here $\mathcal J=\pi_*\mathcal O_S(-2L)$:
$x$ vanishes along $2L$, and the order-two condition on $C'$
forces any function in the push-forward to be divisible by $x$
on the normal surface $V$. On $S$ we have
\[
 \mathcal J\mathcal O_S
 =\mathcal O_S(-2L-2E),\qquad
 Q=2E=2K_{S/V}.
\]
The generation defect is $g_{S,2}(\ord_E)=2$, but
\eqref{eq:relative-defect-comparison} gives $H=0$. By
\eqref{eq:lci-threshold-identity},
\[
 A_{\mathrm{MJ}}(w;V,(x)^{1/2})=A(w;S,L)
\]
for every divisorial valuation $w$. Moreover,
$\mathfrak a\mathcal O_S=\mathcal O_S(-F'-E)$.
The divisor $L+t(F'+E)$ has simple normal crossings support
and coefficients at most one precisely for $0\leq t\leq1/2$.
It follows that the two log canonical thresholds over the vertex
are both $1/2$.

This example arises from an algebraic separatrix as follows.
On the smooth threefold $W=S\times\mathbb A^1_t$, the regular foliation
$\ker(dt)$ has invariant divisor $S\times\{0\}$, whose
image under $\pi\times\mathrm{id}$ is
$V\times\{0\}\subset V\times\mathbb A^1_t$.
The divisor $L$ is the restriction of the transverse boundary
pulled back from $S$.
\end{Example}

\subsection{Multiplier ideals on normalisations and the conductor}

Let $Z=Z_X^{\mathrm{sn}}$, with finite normalisation
$\nu:\widetilde Z=\coprod V_i\to Z$, and fix an ideal
$\mathfrak a$ on $X$ non-zero on each branch.
For $t\geq0$, let
$\mathcal J_i(t)=\mathcal J_{\mathrm{MJ}}
 (V_i,\mathfrak b_i;(\mathfrak a\mathcal O_{V_i})^t)$.
Assume these are ideals in $\mathcal O_{V_i}$, as in the effective range.
Define
\begin{equation}\label{eq:multiplier-equaliser}
\mathcal J_Z(t)
 =\mathcal O_Z\cap\nu_*\left(\bigoplus_i\mathcal J_i(t)\right)
 \subset\nu_*\mathcal O_{\widetilde Z}.
\end{equation}
The sections satisfy the conductor gluing condition even when
the ideals on distinct branches have different restrictions.

\begin{Proposition}\label{prop:coherence}
The equaliser \eqref{eq:multiplier-equaliser} is a coherent
ideal of $\mathcal O_Z$. It is the unit ideal near a specified fibre
if and only if every $\mathcal J_i(t)$ is the unit ideal near the
inverse image of that fibre. The construction depends on the chosen
Mather--Jacobian pairs on the normalised image branches.
\end{Proposition}
\begin{proof}
Finite push-forward preserves coherence. The intersection in
\eqref{eq:multiplier-equaliser} is the kernel of the
natural morphism from $\mathcal O_Z$ to the finite coherent module
$\nu_*\bigoplus_i(\mathcal O_{V_i}/\mathcal J_i(t))$.
It is therefore coherent. It contains $1$ at a point precisely
when every branch ideal contains $1$ at each point of its finite
inverse image. The claim on neighbourhoods follows because the
image of the support of the quotient under a finite morphism is
closed.
\end{proof}

\begin{Proposition}
Suppose the specified branch pairs are MJ-klt near the fibre over
$Y$, and let $\tau$ be the minimum of their MJ log canonical
thresholds for $\mathfrak a$. Suppose also that the additional
conductor strata, if included in the chosen finite image system,
have thresholds at least $\tau$. Then the minimum threshold of
that image system equals $\tau$; the equaliser
\eqref{eq:multiplier-equaliser} is the unit ideal there for
$t<\tau$ and is proper for $t>\tau$.
\end{Proposition}
\begin{proof}
On each MJ-klt branch, the multiplier ideal ceases to be the
unit ideal when the weight crosses its log canonical threshold.
Thus $\tau$ is the first such threshold among the branches. Proposition
\ref{prop:coherence} transfers precisely this property to
the equaliser. Taking the minimum over any additional specified
conductor strata does not change $\tau$ under the stated
hypothesis.
\end{proof}
A coefficient-one conductor can give a non-trivial multiplier
ideal already at $t=0$; the MJ-klt and conductor-threshold
hypotheses above cannot simply be discarded. These thresholds are
those of the specified Mather--Jacobian pairs on the normalised
images. Comparing them with a foliated threshold on an adapted
threefold additionally requires the discrepancy comparison in
\eqref{eq:complete-mj-comparison} and agreement of the ideals whose
orders enter the threshold.

\begin{Proposition}
Let the normalised image branches $V_i$ carry effective boundaries
for which the ordinary Mather--Jacobian discrepancy criteria apply.
If the branch pairs are MJ-canonical, the $V_i$ have rational
singularities. If the branch pairs are MJ-log canonical, the $V_i$
are Du Bois. In the latter case, suppose also that the conductor
$C\subset Z$ and its inverse image
$\widetilde C\subset\widetilde Z$ are Du Bois. Then $Z$ is Du Bois.
\end{Proposition}
\begin{proof}
Dropping an effective boundary can only increase
Mather--Jacobian discrepancies. Apply the ordinary rationality
and Du Bois criteria to the normal branches
\cite[Theorem~7.7]{dFD}.

For the final assertion, recall that $Z=Z_X^{\mathrm{sn}}$ is
seminormal. Its conductor is radical in both $\mathcal O_Z$ and
$\nu_*\mathcal O_{\widetilde Z}$. Indeed, locally write these rings
as $A\subseteq B$ and the conductor as $\mathfrak c$. If
$b^2\in\mathfrak c$ for $b\in B$, then for every $z\in B$ both
$(bz)^2$ and $(bz)^3$ lie in $A$. Seminormality gives $bz\in A$,
so $b\in\mathfrak c$. Thus $C$ and $\widetilde C$ are reduced,
and the conductor square gives the exact sequence
\[
0\longrightarrow\mathcal O_Z
\longrightarrow\nu_*\mathcal O_{\widetilde Z}\oplus\mathcal O_C
\longrightarrow\nu_*\mathcal O_{\widetilde C}
\longrightarrow0.
\]
The Du Bois gluing triangle for this proper square compares the
corresponding Du Bois complexes. The natural maps for
$\widetilde Z$, $C$, and $\widetilde C$ are quasi-isomorphisms
by hypothesis. Since $\nu$ is finite, its push-forward is exact
on coherent sheaves, and comparison with the displayed sequence
shows that $\mathcal O_Z\to\underline\Omega_Z^0$ is also a
quasi-isomorphism \cite[Theorem~3.1(iv)]{SchwedeDB}.
\end{proof}

\section{Examples}

In the following examples we write $qE$ for
$q\operatorname{ord}_E$.

\subsection{Regular, dicritical and second-type examples}

For $\mathcal F=\ker dx$ on $\mathbb A^n$, an invariant coordinate
divisor has $\codim C_X(qE)=q$ and $t_{\mathcal F}(qE)=q-1$,
giving $A_{\mathcal F}(E)=0$. A transverse coordinate divisor
has contact zero and discrepancy one, or $1-c$ with boundary
coefficient $c$.

For $\omega=x\,dy-y\,dx$ on $\mathbb A^3$, blowing up the
singular axis gives $\pi^*\omega=x^2du$ in $y=xu$.
The exceptional $E$ is transverse, $a(E)=-1$, and
$A_{\mathcal F}(E)=0$; $K_{\mathcal G}+E=\pi^*K_{\mathcal F}$.
The foliation is log canonical \cite[Lemma~3.3]{CSMMP}.
Arcs $x=t^qa$, $y=t^qb$ give
\[
x\,dy-y\,dx=t^{2q}(ab'-ba')dt,\qquad
t_{\mathcal F}(qE)=\codim C_X(qE)=2q.
\]
Exact tangential arcs satisfy $y=cx$ or $x=0$. These
$\mathbb P^1$-varying leaves admit no finite separatrix divisor;
they lift to $u=c$, with multiple lifts for axis arcs.
The crepant leaf boundary gives logarithmic codimension $q-q=0$.

For $\omega=y^2dx-x^2dy$, the same blow-up gives
$x^2\{(u^2-u)dx-x\,du\}$. Now $E$ is invariant, saturation has
order two, and $A_{\mathcal F}(E)=-1$. The differential order is $3q-1$ and the codimension is $2q$.
The multiplicity-one difference vanishes, whereas the formula
comparing $q$ and $2q$ gives $2+2-5=-1$.

\subsection{Resonance and the second simple type}
For $\omega=y\,dx-\lambda x\,dy$, the toric weight $(a,b,0)$
is dicritical precisely when $a-\lambda b=0$; this is the
resonant case of Theorem~\ref{thm:resonance-dicritical}.
The differential-contact formula gives discrepancy zero in
both cases and distinguishes their invariance.
The toric comparison fails for other resonant valuations.
For $\omega=y\,dx-x\,dy$ and $D=(xy=0)$ on $\mathbb A^2$,
blow up the origin: $x=u$, $y=uv$ gives
$\mu_1^*\omega=-u^2\,dv$ and a dicritical $E_1$.
Blow up $(u,v)=(0,1)$, taking $u=s$, $v=1+sw$:
\[
(\mu_1\circ\mu_2)^*\omega
=-s^2(w\,ds+s\,dw).
\]
Here $E_2=(s=0)$ is invariant, with saturation order $2$,
$\ord_{E_2}(D)=2$, and $A_W(E_2)=3$. Hence
$A_{\mathcal F}(E_2)=3-1-2=0$ but $A(E_2;W,D)=1$.
Non-resonance in Theorem~\ref{thm:ordinary-pair} cannot be dropped.

For $\omega=y\,dx-x^2dy$, one has
$\pi^*\omega=x\{u(1-x)dx-x^2du\}$, contact $2q-1$,
and $A_{\mathcal F}(E)=0$. It is canonical \cite[Lemma~2.9]{CSMMP}. On the weak branch,
$\Theta_y=2C$ and $R_y=L_y=C$, where $C=(x=0)|_{S_y}$.
The actual branch discrepancy $-1$ becomes $0$ after adding
$\ord_C R_y=1$; the cylinder formula is $q-2q+q=0$.
On the strong branch $\Theta_x=(y=0)$ and $R_x=0$.

\subsection{A singular ambient variety}
Let $V_l=\mathbb A^2/\mu_l$, for the scalar action on $(a,b)$,
and $X_l=V_l\times\mathbb A^1_z$. Write $x_j=a^{l-j}b^j$.
On the standard resolution chart $x_j=st^j$, $E=(s=0)$.
Since $dx_0\wedge dx_1\wedge dz=s\,ds\wedge dt\wedge dz$
has order one, as do the minimal top differential minors,
$\widehat k_E(X_l)=1$ and $\codim C_{X_l}(qE)=2q$.
Descend the foliation $\ker da$. The leafwise form $db\wedge dz$
has character one, so $lK_{\mathcal F}$ is Cartier. On the
resolution the foliation is $\ker ds$ and $db=s^{1/l}dt$
along its leaves, giving $a(E;\mathcal F)=A_{\mathcal F}(E)=1/l$.
On the cyclic cover the top differential contributes $s^l$ to
the $l$-th tensor and the inverse canonical factor contributes
$s^{-1}$; thus $\mathcal P_l$ is generated generically by
$s^{l-1}(ds)^{\otimes l}$. Consequently
\[
\vartheta_{\mathcal F,0}(qE)=\left(2-\frac1l\right)q-1,
\qquad 2q+\vartheta(qE)-\vartheta(2qE)=\frac ql.
\]
For $l=2$, $V_2=(x_0x_2-x_1^2=0)$ has Jacobian order one,
$\mathcal Q_2=(dx_0)^{\otimes2}/x_0$, and $\rho(qE)=q/2-1$;
\eqref{eq:lci-formula} gives $q/2$.
For $l=3$, the defining minors
$x_0x_2-x_1^2$, $x_0x_3-x_1x_2$, $x_1x_3-x_2^2$
have Jacobian order two. Here
$\mathcal Q_3=(dx_0)^{\otimes3}/x_0^2$, $\rho(qE)=q/3-1$,
and $\ord_E\mathfrak n_{3,X_3}=4$.
The Nash formula gives $2q-4q/3-q/3=q/3$,
whereas replacing the Nash correction by the Jacobian order
gives $2q-2q-q/3=-q/3$.

\subsection{The fixed-divisor criterion}
On $X_l$ above, descend the integrable form
$\omega=b^2da-a^2db$. Its conormal and ambient canonical
characters are three and two, so its foliated canonical
character is $-1$ and its Cartier index is $l$. On the
cyclic cover the resolution chart gives
\[
\omega=\tfrac1l s^{3/l-1}\bigl((t^2-t)ds-ls\,dt\bigr).
\]
The bracketed form $\alpha$ defines an invariant $E$.
The top differential has factor $s^{2/l-1}$; its quotient
by $\alpha$, raised to the $l$-th power, has factor $s^{-1}$.
Thus $A_{\mathcal F}(E)=-1/l$ and $\mathcal P_l$ is generated
generically by $s^{l+1}\alpha^{\otimes l}$. Consequently,
\[
\vartheta(qE)=\left(2+\frac1l\right)q-1,\qquad
\codim C_{X_l}(qE)-\vartheta(qE)=1-\frac ql,
\]
which changes sign at $q=l+1$, proving the fixed-divisor
sharpness of \eqref{eq:intro-index}.

\subsection{Branch contacts and conductor strata}

\begin{Example}

For $D=(xy=0)$ and $B=a(z=0)$, the branch $S_x$ has boundary
$(y=0)+a(z=0)$. Contact orders $(m,n)$ give codimension $m+n$,
boundary order $m+an$ and logarithmic codimension $(1-a)n$.
Formula~\eqref{eq:elementary-discrepancy} also gives:
\[
\begin{array}{c|ccccc}
Z&S_1\cap S_2&S_1\cap T_1&S_1\cap T_1\cap T_2&S_1\cap S_2\cap T_1&S_1\cap S_2\cap S_3\\
\hline
A_{\mathcal H}(E)&0&1-a_1&2-a_1-a_2&1-a_1&0.
\end{array}
\]
At a triple point, $\Theta_x-C_{xy}=C_{xz}$ and
$\Theta_y-C_{xy}=C_{yz}$ restrict to the same reduced point on
$C_{xy}$, counted once. The weights $(1,2,2)$ induce $2(1,1)$
on $S_x$, with ambient discrepancy $2(1-a)$ versus primitive branch
value $1-a$.
\end{Example}

\begin{Example}\label{ex:mj-cubic-cone}
For the rational normal cubic cone $V_3$, the resolution
$Y=\operatorname{Tot}\mathcal O_{\mathbb P^1}(-3)$ has
$(x_0,x_1,x_2,x_3)=(s,st,st^2,st^3)$ and $F=(s=0)$.
Mather and Jacobian orders are $1$ and $2$, giving
$A_{\mathrm{MJ}}(F;V_3)=0$ and
$\operatorname{codim}N_q(F)-q\operatorname{ord}_F\mathfrak j_{V_3}=0$.
Since $F^2=-3$ and $K_Y\cdot F=1$, ordinary discrepancy is $2/3$,
not $1$ as on a smooth zero-boundary model. More precisely, if
$S=Y$ and $L=0$, then the image ideal and its contraction are both
$\mathcal O_{V_3}$, so $g_{S,r}(\operatorname{ord}_F)=e(\operatorname{ord}_F)=0$.
Nevertheless $A_{\mathrm{MJ}}(F;V_3)=0$ whereas $A(F;S)=1$.
Thus the first difference on the right of
\eqref{eq:complete-mj-comparison} is $-1$, independently of the
boundary ideals. The cone occurs as
$V_3\times\{0\}$ in the projection foliation on
$V_3\times\mathbb A^1$.
\end{Example}

\begin{Example}
For $V_2=(uv-w^2=0)$, the resolution chart
$(u,v,w)=(s,st^2,st)$ has Mather and Jacobian orders $1$ along
$F=(s=0)$. Hence $A_{\mathrm{MJ}}(F;V_2)=A(F;V_2)=1$;
corrected arc codimension is $q$, rather than $2q$.
Adjunction comparison still requires a crepant boundary.
\end{Example}

\subsection{A threshold on a weak separatrix}

\begin{Example}\label{ex:weak-threshold}
Let
$
 \omega_\ell=y\,dx-x^{\ell+1}dy,\
 S=(y=0),\ D=(xy=0),\
 \mathfrak a=(x^a+y,z^b)
$
on $\mathbb A^3_{x,y,z}$, where $\ell,a,b\geq1$ are integers.
On the weak separatrix,
$L=C$, $\Theta=(\ell+1)C$ and $R=\ell C$, with
$C=(x=0)|_S$. The foliated log canonical threshold at the origin is
\[
 \operatorname{lct}_{\mathcal H}(\{0\};\mathbb A^3,0;\mathfrak a)
 =\frac1b.
\]
Indeed, every divisorial valuation $w$ of $S$ centred at the origin
satisfies
$
 A(w;S,C)\geq w(z),\
 w(x^a,z^b)\leq b\,w(z).
$
The first inequality follows from the regular-parameter estimate
on the smooth surface. Therefore every quotient in
\eqref{eq:threshold-adjunction} is at least $1/b$.
For any integer $m\geq1$ with $am\geq b$, the primitive monomial
valuation $w(x)=m$, $w(z)=1$ attains this bound.
For every integer $N>am$, the primitive ambient weights
$(m,N,1)$ define an invariant divisor $E_N$. Theorem
\ref{thm:ordinary-pair} gives
$
 A_{\mathcal H}(E_N)=1,\
 \ord_{E_N}\mathfrak a=b,\
 \ord_{E_N}S=N.
$
Thus these divisors all compute the threshold. On the branch,
\[
 A(w;S,\Theta)=1-\ell m,\qquad w(R)=\ell m.
\]
Their sum is the numerator $1$ in
\eqref{eq:threshold-adjunction}. The threshold depends only on $b$.
\end{Example}

\paragraph{\textbf{Acknowledgements.}}
The author thanks the Universit\`a degli Studi di Bari Aldo Moro
and INdAM--GNSAGA for support.

\paragraph{\textbf{AI Declaration}}
The mathematical ideas, constructions, results and arguments presented
in this paper are the author's own.
ChatGPT (GPT-6) was used for proofreading, improving the exposition,
checking computations, and providing mathematical assistance, particularly
in the development and refinement of the proof of
Lemma~\ref{lem:normal-cartier-lift}, concerning divisorial lifting
from normal Cartier surfaces.
The author takes full responsibility for the mathematical content.

\end{document}